\documentclass[11pt,twoside,reqno,centertags]{amsart}

\thanks{AMS Subject Classifications: 86A10, 34B20}

 \usepackage{amsmath,amsthm,amsfonts,amssymb}
\usepackage[mathscr]{euscript}
 \usepackage[hidelinks]{hyperref} 
 \usepackage{mathrsfs} 
\usepackage{graphicx}
\usepackage{color}
  \usepackage{epsfig}
  \usepackage{enumitem}
  \usepackage{csquotes}
  \usepackage{mathtools}
 
 \newtheorem{theorem}{Theorem}
 \newtheorem{lemma}[theorem]{Lemma}
 \newtheorem{proposition}[theorem]{Proposition}
 \theoremstyle{definition}
 \newtheorem{remark}[theorem]{Remark}
 
 \newcommand{\N}{\mathbb N}
 \newcommand{\R}{\mathbb R}
 \newcommand{\C}{\mathbb C}
 \newcommand{\dd}{\mathrm d}
 \DeclareMathOperator{\Real}{Re}
  
\begin{document}
 
\title{On the propagation of mountain waves: Linear theory}
\date{}
\maketitle     
 
\vspace{ -1\baselineskip}

{\small
\begin{center}
 {\sc Adrian Constantin} and {\sc J\"org Weber} \\
Faculty of Mathematics, University of Vienna, AT-1090 Vienna, Austria
\end{center}
}

\numberwithin{equation}{section}
\allowdisplaybreaks

\smallskip

 \begin{quote}
\footnotesize
{\bf Abstract.}  
We derive and establish a solution concept for the linear mountain wave problem in two dimensions. After linearizing the governing equations and a change of variables, the problem can be stated as a Dirichlet boundary value problem for a Helmholtz equation in terms of the vertical wind profile in the upper half-plane, with altitude-dependent potential (the Scorer parameter). To single out the correct solution, we have to make use of a radiation condition which is, due to the different physical situation, different from the classical Sommerfeld radiation condition for electromagnetic or acoustic waves. We rigorously develop a transform method and construct the physically correct solution, following Lyra's monotonicity criterion for mountain waves. In this procedure, we clearly recognize the two typical types of mountain waves: vertically propagating waves and trapped lee waves. This paper is the first rigorous work on Helmholtz-like equations in the upper half-plane subject to such a non-classical radiation condition.
\end{quote}

\section{Introduction}

Atmospheric gravity waves, influenced by air compressibility and stratification, can propagate both vertically and horizontally. A common example is mountain waves, which form when strong, stable winds flow over a mountain range, 
provided that the wind direction has a sizable component which is perpendicular to the range. Some mountain waves propagate over long distances (up to tens of km vertically, well into the stratosphere, and hundreds of km horizontally downwind), 
while in others the streamlines steepen and overturn like breaking ocean waves, often producing significant turbulence (see Fig.~\ref{fig-b}). The US 
National Transportation Safety Board records from 1990 to 2017 list 42 accidents in which mountain wave turbulence was a primary contributing factor. 
While turbulence is a well-acknowledged hazard to aviation, the mere encounter with a mountain wave can be hazardous for aircrafts. For example, in May 2023 a small Cessna 182 airplane crashed in Alaska, 
with the pilot and the passenger sustaining only minor injuries, after encountering 
at an altitude of about 4 km mountain wave conditions with downdraft speeds exceeding $11\;\mathrm{m\,s^{-1}}$ (much larger than the climb rate of the aircraft). 

The meteorological literature distinguishes between two types of mountain waves: vertically propagating waves and trapped lee waves. For both wave types viscosity effects are not relevant, in contrast to the rotor circulations 
produced by the fast near-surface downslope winds on the lee side of the mountains, which are inherently viscous vortex phenomena located below the elevation of the mountaintops. Throughout this paper we will only discuss inviscid flow patterns. 
Vertically propagating mountain waves occur when static stability increases above the mountain peak and when the wind speed does not increase significantly with height. 
Since the air density decreases with height, these waves amplify with altitude and can reach the lower stratosphere (the bottom of which is about 10 km above the ground at mid-latitudes), unless they break first, causing strong turbulence. 
Their vertical velocities can exceed $30\;\mathrm{m\,s^{-1}}$, making them a risk even to large aircraft -- see the case study \cite{bdw} for nonbreaking, vertically propagating mountain waves as a hazard to high-flying aircraft. The waves that do not break in the troposphere or in the stratosphere 
are ultimately dissipated before reaching the mesosphere (which extends from about 50 to 85 km above ground) 
by the vertical transfer of infrared radiation between the warm and cool regions within the wave and the surrounding atmosphere \cite{Dur}. On the other hand, trapped mountain waves occur when wind speed 
above the mountain increases sharply with height, blocking upward wave propagation, or when stability decreases in the layer just above the mountain top (typically due to 
thermal inversion -- a reversal of the normal behavior of temperature in the troposphere, with a warmer layer over cool air). As the wave energy is trapped vertically, these waves will propagate horizontally 
and they can extend hundreds of km downwind of the mountain range that creates them.

\begin{figure}[ht]
	\noindent\includegraphics[width=11cm]{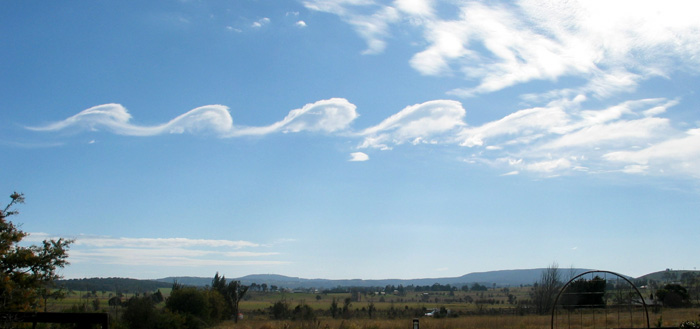}
	\caption{\footnotesize{Breaking mountain waves over Mount Duval in Australia (Source: $\copyright$GRAHAMUK, CC BY-SA 3.0): the wave slopes steepens until the top of the wave overruns the lower part.}}
	\label{fig-b}
\end{figure}

Mountain waves can occur in dry conditions without visible cloud markers, in which case they have to be monitored using infrared satellite imagery -- 
they cause temperature fluctuations and redistribute atmospheric water vapor, whose variations can be detected (see Fig.~\ref{fig-vp}). However, if the air is humid enough, 
mountain waves lead to specific cloud formations downwind: lee wave clouds, cap clouds and lenticular clouds occur as manifestations of trapped waves, while nacreous clouds (also known as mother-of-pearl clouds) 
are associated to vertically propagating waves reaching the stratosphere. 

\begin{figure}[ht]
	\noindent\includegraphics[width=7cm]{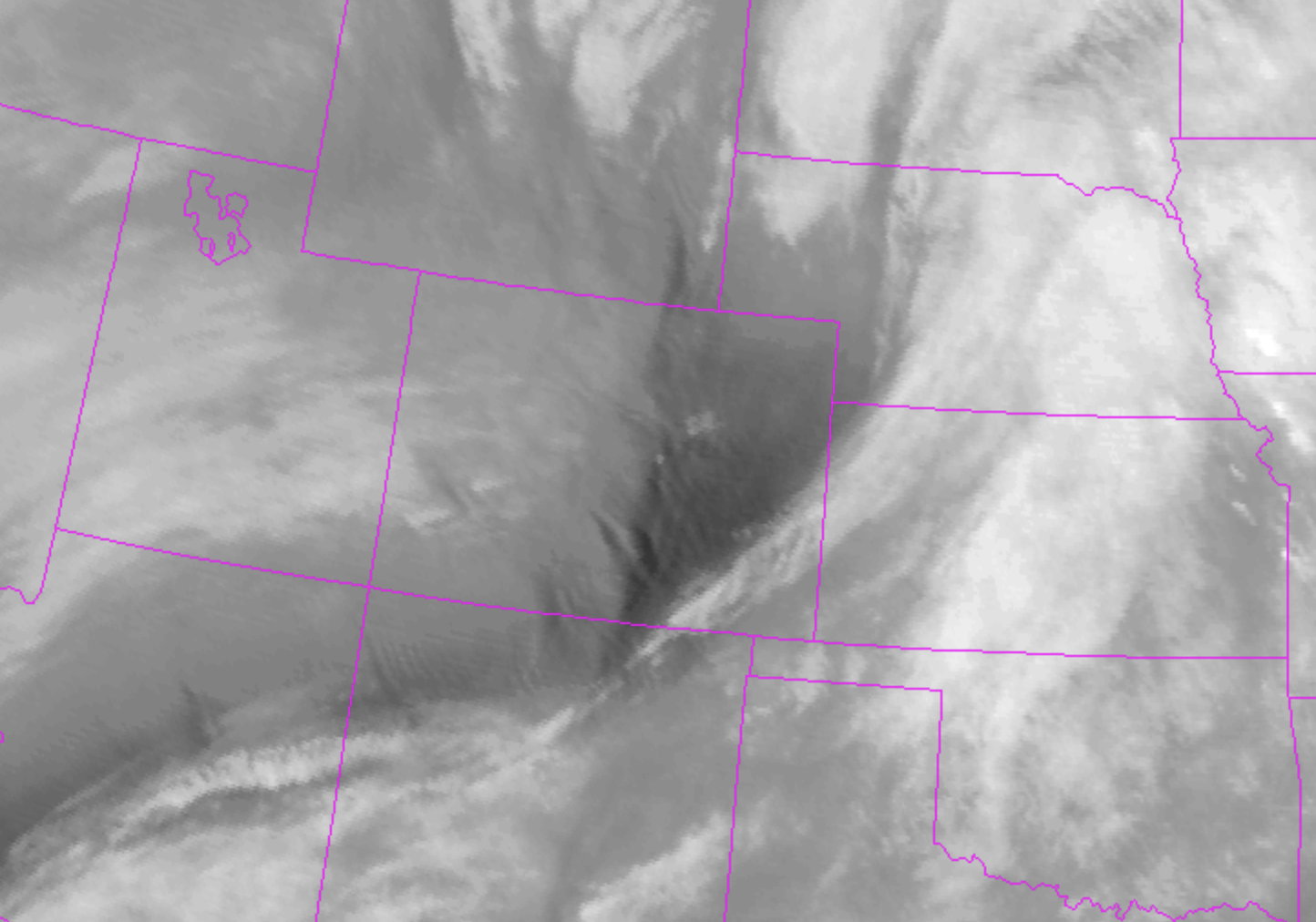}
	\caption{\footnotesize{Water vapor satellite images showing vertically propagating mountain waves over the Rocky Mountains in Colorado by visualizing atmospheric moisture patterns, with 
			relatively dry regions rendered as dark and high-humidity regions shown in white (Source: COMET$\circledR$ https://learn.meted.ucar.edu $\copyright$ 1997--2025). By converting radiation 
			measurements of the effective layer (the highest altitude with appreciable water vapor) into temperature, this imagery can trace upper-level atmospheric flows 
			when cloud condensation does not occur and therefore visible imagery is powerless.}}
	\label{fig-vp}
\end{figure}

When vertically propagating mountain waves enter the stratosphere, they sometimes generate mesmerizing nacreous clouds, well above the more common tropospheric clouds.  
The temperature fluctuations induced by the mountain waves, in conjunction with a very cold stratosphere, play a crucial role in the formation of these clouds: the cold phase of the waves can lower local temperatures sufficiently to trigger 
cloud formation -- in the extremely dry stratosphere, only at very low temperatures (below $-85\,^\circ$C) is there enough condensation to produce clouds. If these clouds are composed of similar-sized crystals, diffraction and interference of light
waves forms generate bright iridescent colors (see Fig.~\ref{fig-nl}). 

\begin{figure} [ht]
	\begin{center}$
		\begin{array}{cc}
			\includegraphics[height=1.54in]{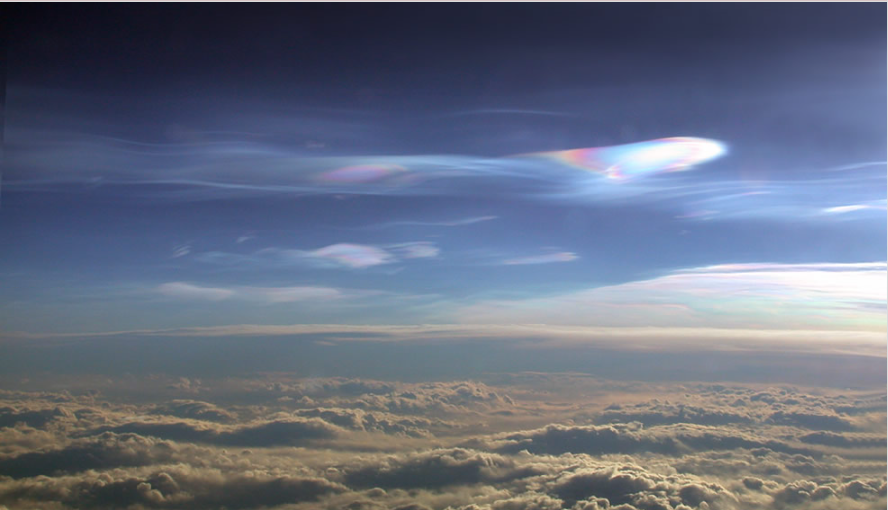} & \includegraphics[height=1.54in]{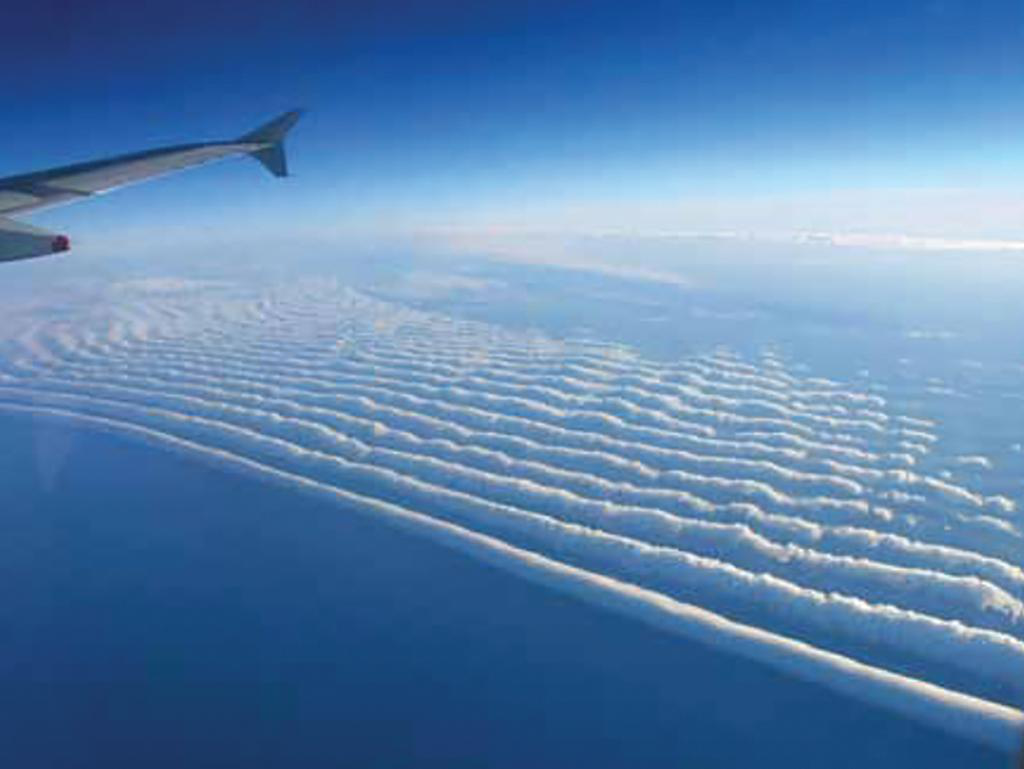}
		\end{array}$
	\end{center}
	\caption{\footnotesize{Left: photograph of  nacreous clouds taken from an aircraft flying in the lower stratosphere, above the highest clouds in the troposphere. 
			The intensely bright wavelike nacreous clouds are higher still (about 20-30 km in the stratosphere). (Source: $\copyright$ Paul A. Newman/NASA).
			Right: aerial photograph of lee wave clouds (Reproduced
			with permission from Comet Program/NOAA.). The lee wave clouds are not carried along with the air flow: they are standing clouds, being continuously formed at
			the leading edge of each wave and continuously eroded at the trailing edge.}}
	\label{fig-nl}
\end{figure}

Lee wave clouds are narrow cloud bands perpendicular to the wind direction and parallel to the mountain chain. The wave length of the trapped mountain waves that produce them ranges from about 3 km to more than 40 km, and the waves 
extend downwind from the mountain range for up to 1000 km. These clouds are confined to a limited altitude band, generally extending a few km above the mountain barrier. Typically a thermal inversion aloft 
acts like a lid that traps the mountain waves and prevents vertical propagation, so that these waves propagate horizontally. The most common inversion is the nocturnal radiational cooling: on clear, calm nights the Earth radiates heat to space, cooling the ground 
and the air just above it, while the air higher up remains warmer. Spectacular undular cloud formations (see Fig.~\ref{fig-nl}) can form by condensation near the top boundary of the inversion
layer: as the warm air holds more moisture, it condenses when it comes into contact with the colder air encountered as the crests of the wave undulations rise, while the wave troughs, where the clouds erode by evaporation, remain below the visible clouds.

A sheared wind profile with wind speed increasing with height above the mountain top can also impede the vertical propagation of mountain waves. In such circumstances, if humidity is high, cap clouds and lenticular clouds may form. 
As the moisture-packed warm air, forced by the wind to rise against the mountain slope, comes into contact with the chilly layer near the mountain top, it starts to cool down. It continues to cool down until it reaches dew point,
and cloud droplets form by condensation near the top of the mountain. If horizontal strong winds just above the mountain top prevent the upward motion of the hitherto rising air, the air flow is guided towards the lee side of the mountain 
and the faster wind above shears off the top of the cloud. On the lee side, the air warms up as it descends down the slope and the cloud dissipates. This sinking air warms by a process called compressional warming, 
since the pressure increases during the descent. As the air sinks and warms, it evaporates the cloud droplets, giving the cloud a smooth edge on the lee side of the mountain. Thus a hat, dome-shape cloud is created 
over the mountain top, this being the trademark of a cap cloud (see Fig.~\ref{fig-lc}). When a layer of dryer air separates two layers of moist air, two cap clouds may form on top of each other, hovering over the same
mountain top. If the wind speed strengthens aloft, but at some height above the mountain top, a similar process leads to the appearance of lenticular clouds having a layered
or stacked shape in the form of a lens or saucer (often mistaken for a UFO). While cap clouds occur directly over a mountain peak, lenticular clouds form on the leeward side of the mountain. 
The wave formed in the wake of the air lifted over the mountain top, due to the interplay between buoyancy and gravity, might lead to the occurrence of several lenticular clouds. 
As the air ascends, it expands and cools. If the air has sufficient moisture, when it reaches the crest of the wave and the temperature drops below dew point, condensation takes
place, allowing for the development of a cloud. The gravitational force resists this vertical motion and the ascending air particles seek to descend. As they do, the air will contract and warm, causing the newly formed cloud to evaporate. 
This up-and-down motion of the air can continue for hundreds of km downwind of the mountain range. Moreover, if there are alternating layers of relatively dry and moist
air aloft, multiple lenticular clouds can be stacked on top of each other (see Fig.~\ref{fig-lc}). Also, for a persistent prevailing wind, the crests and
troughs of the wave of air continue to form downwind, and a series of lenticular clouds may form. What sets lenticular clouds apart from other cloud types is their immobility: most clouds are carried along by winds but lenticular clouds tend to remain in place, 
even for extended periods of several hours, despite being subjected to strong horizontal winds. However, their stationary appearance is deceptive, as these clouds are in a constant state of formation and dissipation: 
they form near the peak of a mountain wave and then disintegrate shortly after passing that point. 

\begin{figure} [ht]
	\begin{center}$
		\begin{array}{cc}
			\includegraphics[height=1.62in]{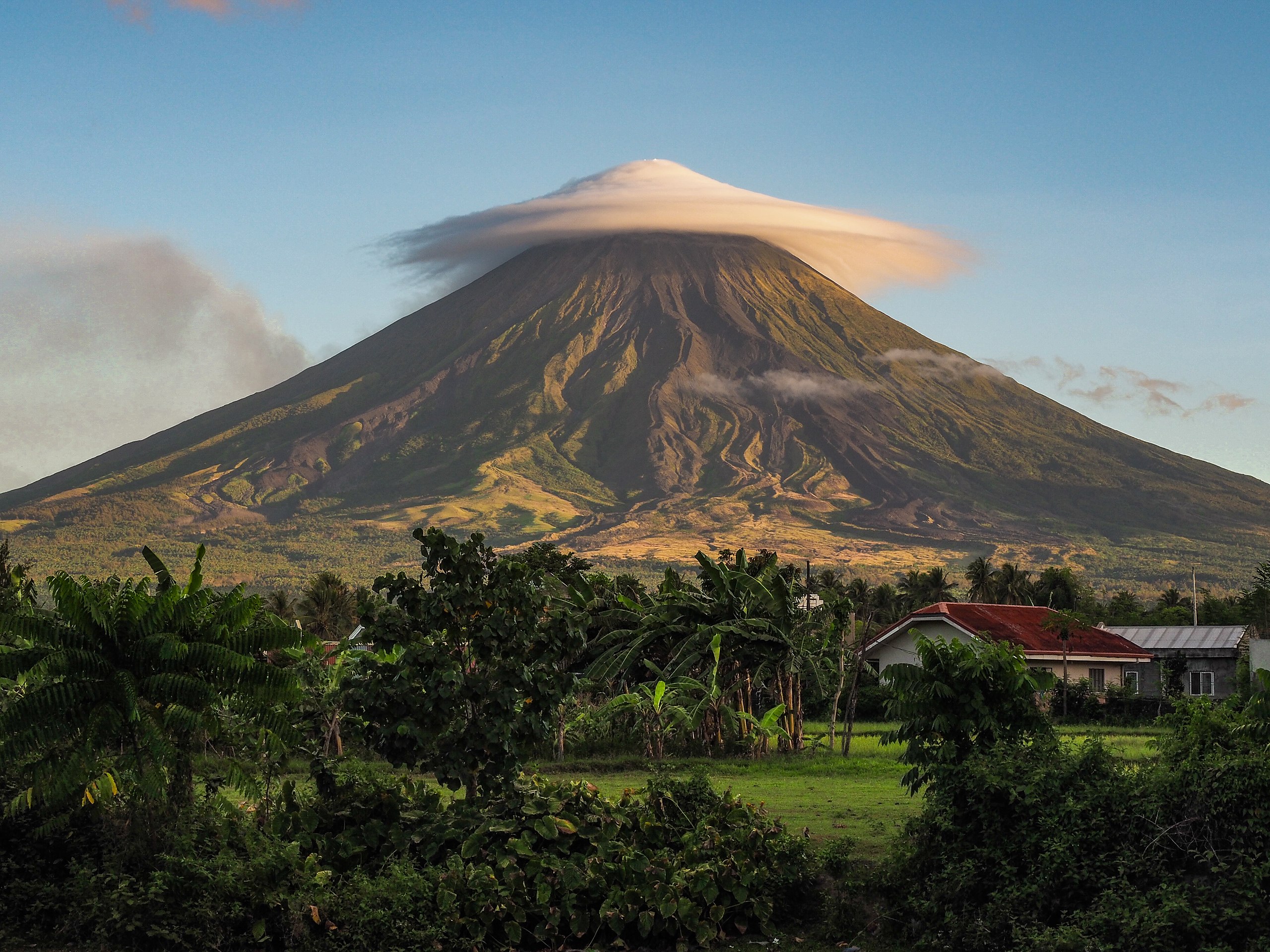} & \includegraphics[height=1.62in]{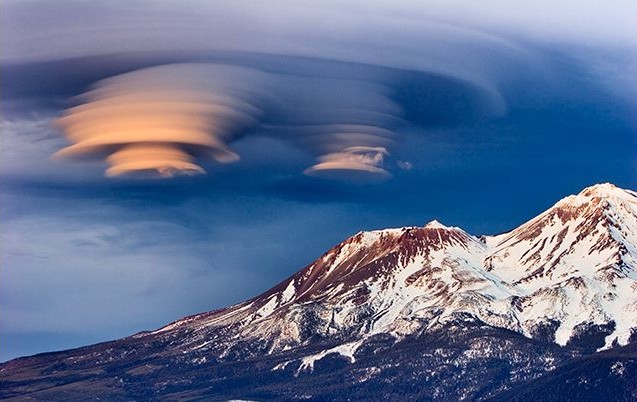}
		\end{array}$
	\end{center}
	\caption{\footnotesize{Left: cap cloud covering the 2462 m high summit crater of Mayon Volcano, Philippines, on 23 April 2019 (Credit: Patryk Reba CC BY-SA 4.0). 
			Right: lenticular clouds over the 4322 m high Mount Shasta in California, on 25 October 2013 (Credit: NOAA).}}
	\label{fig-lc}
\end{figure}

The strongest mountain waves are forced by long ridges that are sufficiently narrow to neglect the dynamical influence of the Coriolis force and to ensure that the setting of two-dimensional inviscid flows is adequate \cite{Dur}. 
To predict the basic structure of a mountain wave, such as its dependence on air stratification, wind speed and mountain profile, one typically relies on linear theory, assuming the mountain to be small in comparison with 
the wavelength of the mountain wave. While nonlinear effects are essential to investigate wave breaking, linear theory is quite accurate for the description of mountain waves that are not near breaking. 
The typical strategy here is to linearize the governing equations, and to subsequently reduce the obtained linear system of equations to a single equation for the 
vertical velocity component -- the physically most significant quantity. This so-called Scorer equation \cite{Scorer49,Smith79} holds in the upper half-plane, subject to boundary conditions depending on the mountain profile.

In principle, the Scorer equation is of Helmholtz-type and therefore its unique solvability calls for a carefully chosen radiation condition. In the context of electromagnetic or acoustic waves, here (versions of) the classical Sommerfeld radiation condition \cite{Sommerfeld1912} comes into play, motivated by the physical principle that such waves are emitted radially outward of a source. This is clearly not physically appropriate for mountain waves, where different patterns are observed 
upstream and downstream. While the question of how to correctly choose a solution based on physical arguments led to some controversy among meteorologists in the last century, by now the radiation condition originally formulated 
by Lyra \cite{Lyra43} has gained widespread acceptance (see the discussion in \cite{Smith79, Smith19}) and we will rely on it.

The motivation of our paper is fourfold: First, typically in the meteorological research 
literature a Scorer equation is found that includes a quite simple Scorer parameter. This, however, relies on various simplifications such as the Boussinesq approximation. We do not make 
any such simplifications and derive a completely general Scorer equation. Moreover, we introduce a change of variables such that the Scorer equation transforms into a Helmholtz-like equation. All of this is done providing explicit formulas of all coefficients in terms of the background state about which we linearize. Second, we develop a transform method based on Weyl--Titchmarsh theory to carefully construct the physically correct solution formula. In particular, we see that the radiation condition dictates which choices in the procedure have to be made. Moreover, in this solution formula one can clearly distinguish between the different types of mountain waves -- vertically propagating waves and trapped lee waves; see Fig.~\ref{fig-sk} and \eqref{eq:formula_w}--\eqref{eq:K^t}. The strength of this transform method is that it works for general Scorer parameters and that it nicely splits the Green's function into different pieces that all have a clear physical interpretation. Third, we provide explicit solution formulas corresponding to some examples of Scorer parameters. A visualization (see Figs.~\ref{fig:KernelConstant} and \ref{fig:KernelMorse}) shows that in one example vertically propagating waves are visible and trapped lee waves absent, while in another example a trapped lee wave is the predominant feature of the linearized flow. We are not aware of other such explicit examples giving rise to trapped lee waves. Fourth, we put our transform method and solution formula on rigorous mathematical ground, a step absent in the 
meteorological research literature. For this we need some assumptions on the Scorer parameter, basically that it approaches a constant at large altitudes, so that Weyl--Titchmarsh and scattering theory for Schrödinger operators on the half line become applicable. Moreover, we provide exact asymptotics for the vertical velocity component upstream which is the mathematical manifestation of Lyra's radiation condition.

\begin{figure}[ht]
	\noindent\includegraphics[width=6cm]{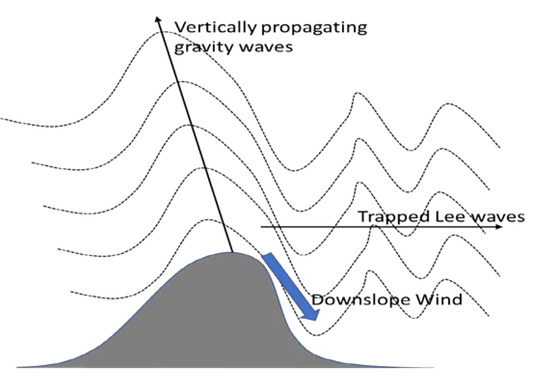}
	\caption{\footnotesize{Sketch of the two main types of mountain waves.}}
	\label{fig-sk}
\end{figure}

Our paper is structured as follows: First, in Section~\ref{sec:equations} we introduce the set of governing equations and their linearized versions, where we consider small perturbations of a horizontal background wind generated by a topographic perturbation. Then, we show, without any simplifications, how to derive the Scorer equation and how to put it in normal form in Section~\ref{sec:Scorer}. Here, by \enquote{normal form} we mean a Helmholtz-like equation with altitude-dependent potential. In particular, we provide explicit, general formulas for the appearing coefficients which reduce, under some approximations, to the formulas typically found in the meteorological research literature. In Section~\ref{sec:derivation}, we 
physically motivate Lyra's radiation condition and discuss how it pertains to our problem. Moreover, we introduce the transform method and construct the physically relevant solution formula. We also provide upper and lower bounds on the number of trapped lee waves depending on the Scorer parameter in Proposition~\ref{prop:trapped}. Then, in Section~\ref{sec:examples} we illustrate our method and solution formula using specific examples. In particular, we consider a constant Scorer parameter, where we recover Lyra's formula, and a Scorer parameter derived from the classical Morse potential. Finally, in Section~\ref{sec:rigorous} we rigorously justify our solution formula, establishing suitable estimates on the Green's kernel, and show (see Theorem~\ref{thm:main}) 
that the vertical velocity component is indeed monotone upstream at fixed altitude, which is Lyra's radiation condition.

\section{The governing equations and their linearization}\label{sec:equations}
Let us denote $\Omega\coloneqq\{(x,z)\in\R^2:z>h(x)\}$, where $z=h(x)$ is the profile of the mountain range (which extends in the $y$-direction without change of shape), with $h$ vanishing at $\pm\infty$. Moreover, by convention, $x=-\infty$ is upstream (windward) and $x=+\infty$ is downstream (leeward). Then, we are concerned with the following equations \cite{Constantin23,Smith79}: The Euler equations
\begin{subequations}\label{eq:system_nonlinear}
	\begin{align}
		u_t + uu_x + wu_z &= -\frac{p_x}{\rho}&\text{in }\Omega,\\
		w_t + uw_x + ww_z &= -\frac{p_z}{\rho} - g&\text{in }\Omega\label{eq:Euler2}\\
		\intertext{for the velocity field $(u,w)$, the pressure $p$ and the density $\rho$; the equation of mass conservation}
		\rho_t + u\rho_x + w\rho_z + \rho(u_x+w_z) &= 0&\text{in }\Omega,\\
		\intertext{the ideal gas law}
		p &= \rho T&\text{in }\Omega\label{eq:gas_law},\\
		\intertext{for the temperature $T$; and the first law of thermodynamics}
		T_t + uT_x + wT_z - \frac{\mu}{\rho}(p_t + up_x + wp_z) &= 0&\text{in }\Omega.\label{eq:thermolaw}\\
		\intertext{Moreover, at the boundary we have the kinematic condition}
		w &= h_xu&\text{on }\partial\Omega.
	\end{align}
\end{subequations}
Here, the equations have already been put into non-dimensional form, to which dimensional quantities (with primes) are related by
\begin{gather*}
	t' = (L'/U') t,\qquad (x',z') = L' (x,z),\qquad (u',w') = U' (u,w),\\
	\rho' = \bar\rho' \rho, \qquad p' = \bar\rho' U'^2 p,\qquad T' = (U'^2/\mathfrak R') T.
\end{gather*}
Above, the typical scales are $L'=2\;\mathrm{km}$, $U'=20\;\mathrm{m\,s^{-1}}$ and $\bar\rho'=1\;\mathrm{kg\,m^{-3}}$, while $\mathfrak R'\approx 287\;\mathrm{m^2\,s^{-2}\,K^{-1}}$ is the gas constant for dry air. Moreover, the non-dimensional constants in \eqref{eq:Euler2} and \eqref{eq:thermolaw} are
\[g = g'L'/U'^2 \approx 49, \qquad \mu = \mathfrak R' / c_p' \approx 0.287,\]
where $g'$ is the gravitational acceleration and $c_p'\approx 1000\;\mathrm{m^2\,s^{-2}\,K^{-1}}$ is the specific heat of dry air at an atmospheric pressure of $1000\;\mathrm{mb}$. Notice that for mountain waves 
outside regions of active precipitation it is reasonable to neglect heat sources (see \cite{Constantin23}), which would otherwise appear on the right-hand side of \eqref{eq:thermolaw}. For more details on this set of equations we refer to \cite{Constantin23,Smith79}. We also mention that the mountain wave problem admits explicit, particular nonlinear Gerstner-type solutions \cite{Constantin23}, also when including additional effects such as the Earth's rotation \cite{Henry24} or precipitation \cite{Lyons25,LyonsMcCarney25}.

Since we are interested in a time-independent linearized problem, we consider small perturbations of a background state $(u_0,\rho_0,p_0,T_0)$, which is independent of $x$, subject to a flat topography. More precisely we plug the ansatz
\begin{align*}
	u(x,z) &= u_0(z) + \varepsilon \bar u(x,z) + \mathcal O(\varepsilon^2),& w(x,z) &= \varepsilon \bar w(x,z) + \mathcal O(\varepsilon^2),\\
	\rho(x,z) &= \rho_0(z) + \varepsilon \bar \rho(x,z) + \mathcal O(\varepsilon^2),&
	p(x,z) &= p_0(z) + \varepsilon \bar p(x,z) + \mathcal O(\varepsilon^2),\\
	T(x,z) &= T_0(z) + \varepsilon \bar T(x,z) + \mathcal O(\varepsilon^2),& h(x) &= \varepsilon \bar h(x) + \mathcal O(\varepsilon^2)
\end{align*}
into the system \eqref{eq:system_nonlinear} and balance the equations, at least formally, up to order $\varepsilon$. First, the background state has to satisfy
\begin{subequations}\label{eq:background}
	\begin{align}
		p_0' &= -g\rho_0 & \text{in }\R^2_+,\\
		p_0 &= \rho_0 T_0 & \text{in }\R^2_+
	\end{align}
\end{subequations}
by \eqref{eq:Euler2} and \eqref{eq:gas_law}, while all other equations trivialize. Here, $\R^2_+\coloneqq\{(x,z)\in\R^2:z>0\}$, and $'\coloneqq\frac{\dd}{\dd z}$ for quantities which only depend on $z$. Notice that \eqref{eq:background} implies the relations
\begin{subequations}\label{eq:background_otherrelations}
	\begin{align}
		\frac{p_0'}{p_0} &= -\frac{g}{T_0},\\
		\frac{\rho_0'}{\rho_0} &= \frac{p_0'}{p_0}-\frac{T_0'}{T_0} = -\frac{g+T_0'}{T_0}.
	\end{align}
\end{subequations}

At order $\varepsilon$, the (linearized) problem reads \cite{Constantin23,Smith79}
\begin{subequations}\label{eq:system_linear}
	\begin{align}
		\rho_0 (u_0 \bar u_x + \bar w u_0') &= -\bar p_x & \text{in }\R^2_+,\label{eq:Euler1_linear}\\
		\rho_0 u_0 \bar w_x &= -\bar p_z - g\bar\rho & \text{in }\R^2_+,\label{eq:Euler2_linear}\\
		u_0 \bar\rho_x + \bar w \rho_0' + \rho_0 (\bar u_x + \bar w_z) &= 0 & \text{in }\R^2_+,\label{eq:mass_conservation_linear}\\
		\bar p &= \bar\rho T_0 + \rho_0 \bar T & \text{in }\R^2_+,\label{eq:gas_law_linear}\\
		\rho_0(u_0 \bar T_x + \bar w T_0') &= \mu (u_0 \bar p_x + \bar w p_0') & \text{in }\R^2_+,\label{eq:thermolaw_linear}\\
		\bar w &= \bar h_x u_0 & \text{on }\partial\R^2_+.\label{eq:kinematic_BC_linear}
	\end{align}
\end{subequations}

\section{The Scorer equation and its normal form}\label{sec:Scorer}
We now show that \eqref{eq:system_linear} can be viewed as a boundary-value problem for a Helmholtz-like equation on $\R^2_+$, where the derivation is to some extent similar as in \cite{Constantin23,Scorer49,Smith79}. Here, typically, for the sake of simplicity the static stability
\begin{subequations}\label{eq:beta}
	\begin{align}
		\beta&\coloneqq\frac{T_0'}{T_0}-\mu\frac{p_0'}{p_0}=-g(1-\mu)\frac{\rho_0}{p_0}-\frac{\rho_0'}{\rho_0}=\frac{g\mu+T_0'}{T_0}\\
		\intertext{with}
		\beta'&=-g(1-\mu)\frac{\rho_0'}{p_0}-g^2(1-\mu)\frac{\rho_0^2}{p_0^2}-\frac{\rho_0''}{\rho_0}+\frac{\rho_0'^2}{\rho_0^2}
	\end{align}
\end{subequations}
is introduced, using \eqref{eq:background}. Differentiating \eqref{eq:gas_law_linear} with respect to $x$ and combining the result with \eqref{eq:Euler1_linear}, \eqref{eq:mass_conservation_linear}, and \eqref{eq:thermolaw_linear} we get
\begin{subequations}\label{eq:linear_px_rhox}
	\begin{align}
		\bar\rho_x&=\frac{\left(\rho_0'u_0-\rho_0u_0'+\frac{\beta p_0}{(1-\mu)u_0}\right)\bar w +\rho_0u_0\bar w_z}{\frac{p_0}{(1-\mu)\rho_0}-u_0^2},\\
		\bar p_x&=\frac{(\rho_0'u_0-\rho_0u_0'+\beta\rho_0u_0)\bar w +\rho_0u_0\bar w_z}{1-(1-\mu)\frac{\rho_0}{p_0}u_0^2}.
	\end{align}
\end{subequations}
Notice that the denominators here do not vanish since typically $u_0\approx 1$, $\rho_0/p_0\le 10^{-2}$. From this, we can now derive an equation for only $\bar w$,
\begin{align}\label{eq:PDE_bar_w_original}
	\bar w_{xx}+A\bar w_{zz}+B\bar w_z+C\bar w=0 \qquad \text{in }\R^2_+,
\end{align}
with $A$, $B$, $C$ being expressions in terms of the background state. Indeed, differentiating \eqref{eq:Euler2_linear} with respect to $x$ and inserting \eqref{eq:linear_px_rhox} yields
\begin{align*}
	&\rho_0u_0\bar w_{xx}\\
	&=-\bar p_{xz}-g\bar\rho_x\\
	&=-\frac{(\rho_0''u_0-\rho_0u_0''+\beta'\rho_0u_0+\beta\rho_0'u_0+\beta\rho_0u_0')\bar w+(\rho_0'u_0-\rho_0u_0'+\beta\rho_0u_0)\bar w_z}{1-(1-\mu)\frac{\rho_0}{p_0}u_0^2}\\
	&\quad\,-\frac{(\rho_0'u_0+\rho_0u_0')\bar w_z+\rho_0u_0\bar w_{zz}}{1-(1-\mu)\frac{\rho_0}{p_0}u_0^2}\\
	&\omit\hfill$\displaystyle{}-\frac{(1-\mu)((\rho_0'u_0-\rho_0u_0'+\beta\rho_0u_0)\bar w +\rho_0u_0\bar w_z)\left(\frac{\rho_0'}{p_0}u_0^2-\frac{\rho_0p_0'}{p_0^2}u_0^2+2\frac{\rho_0}{p_0}u_0u_0'\right)}{\left(1-(1-\mu)\frac{\rho_0}{p_0}u_0^2\right)^2}$\\
	&\quad\,-\frac{g(1-\mu)\rho_0}{p_0}\frac{\left(\rho_0'u_0-\rho_0u_0'+\frac{\beta p_0}{(1-\mu)u_0}\right)\bar w+\rho_0u_0\bar w_z}{1-(1-\mu)\frac{\rho_0}{p_0}u_0^2}.
\end{align*}
This, after collecting the terms, using \eqref{eq:background}, \eqref{eq:background_otherrelations}, and \eqref{eq:beta}, and simplifying, is \eqref{eq:PDE_bar_w_original} with
\begin{subequations}\label{eq:ABC}
	\begin{align}
		A\coloneqq{}&\frac{1}{1-(1-\mu)\frac{\rho_0}{p_0}u_0^2} = \frac{1}{1-(1-\mu)u_0^2T_0^{-1}},\label{eq:A}\\
		B\coloneqq{}& A^2\left(\frac{\rho_0'}{\rho_0}+\frac{2(1-\mu)\rho_0u_0u_0'}{p_0}+\frac{g(1-\mu)\rho_0^2u_0^2}{p_0^2}\right)\\
		={}& A^2\left(\frac{-g-T_0'+2(1-\mu)u_0u_0'}{T_0}+\frac{g(1-\mu)u_0^2}{T_0^2}\right),\nonumber\\
		C\coloneqq& -A^2\left(\frac{\rho_0'u_0'}{\rho_0u_0}+\frac{u_0''}{u_0}+\frac{g\rho_0'}{\rho_0u_0^2}+\frac{2(1-\mu)\rho_0u_0'^2}{p_0}-\frac{(1-\mu)\rho_0u_0u_0''}{p_0}\right.\nonumber\\
		&\;\left.{}+\frac{2g(1-\mu)\rho_0u_0'}{p_0u_0}+\frac{g^2(1-\mu)\rho_0}{p_0u_0^2}+\frac{g(1-\mu)\rho_0^2u_0u_0'}{p_0^2}+\frac{g^2\mu(1-\mu)\rho_0^2}{p_0^2}\right)\label{eq:C}\\
		=&-A^2\left(\frac{u_0''}{u_0}+\frac{(1-\mu)(2u_0'^2-u_0u_0'')}{T_0}-\frac{u_0'(g+T_0'-2g(1-\mu))}{u_0T_0}\right.\nonumber\\
		&\qquad\quad\left.{}-\frac{g(g\mu+T_0')}{u_0^2T_0}+\frac{g(1-\mu)(g\mu+u_0u_0')}{T_0^2}\right).\nonumber
	\end{align}
\end{subequations}
Here, for convenience, the respective first expressions are given in terms of $(u_0,\rho_0,p_0)$, while the second expressions are in terms of $(u_0,T_0)$. Given observational data, the computations in terms of $(u_0,T_0)$ are likely preferably numerically since $T_0$ remains, with increasing $z$, typically at the same order of magnitude, while $\rho_0$ and $p_0$ decay exponentially.

Returning now to the general formulas without simplifications, we now want to write \eqref{eq:PDE_bar_w_original} in normal form, that is, as a Helmholtz-like equation
\begin{equation}\label{eq:PDE_chi_Helmholtz}
	\chi_{xx}+\chi_{\zeta\zeta}+F\chi=0 \qquad \text{in }\R^2_+,
\end{equation}
where $\chi$ is related to $\bar w$ and $\zeta\in(0,\infty)$ to $z$ by a change of variables. To this end, we clearly simply have to put the ordinary differential operator $A\frac{\dd^2}{\dd z^2}+B\frac{\dd}{\dd z}+C$ into Liouville's normal form. Thus, we introduce the new independent variable
\[\zeta\coloneqq\int_0^z\frac{\dd z'}{\sqrt{A(z')}}.\]
Therefore,
\[\frac{\partial}{\partial z}=A^{-1/2}\frac{\partial}{\partial\zeta},\]
so that
\begin{align*}
	\bar w_z&=A^{-1/2}\bar w_\zeta,\\
	\bar w_{zz}&=A^{-1}\bar w_{\zeta\zeta}-\frac12A^{-3/2}A'\bar w_\zeta.
\end{align*}
Inserting this in \eqref{eq:PDE_bar_w_original} yields
\begin{align}\label{eq:PDE_bar_w_step1}
	0=\bar w_{xx}+A\bar w_{zz}+B\bar w_z+C\bar w=\bar w_{xx}+\bar w_{\zeta\zeta}+D\bar w_\zeta+C\bar w
\end{align}
with
\begin{align}\label{eq:D}
	D\coloneqq A^{-1/2}\left(B-\frac12A'\right).
\end{align}
Now introduce the new dependent variable
\[\chi\coloneqq E\bar w\]
with
\begin{equation}\label{eq:E}
	E(\zeta)\coloneqq\exp\int_0^\zeta\frac{D(\zeta')}{2}\,\dd\zeta'.
\end{equation}
We have
\begin{align*}
	\bar w_\zeta&=E^{-1}\left(\chi_\zeta-\frac D2\chi\right),\\
	\bar w_{\zeta\zeta}&=E^{-1}\left(\chi_{\zeta\zeta}-D\chi_\zeta+\left(\frac{D^2}{4}-\frac{D_\zeta}{2}\right)\chi\right).
\end{align*}
Inserting this in \eqref{eq:PDE_bar_w_step1} (multiplied by $E$) finally yields \eqref{eq:PDE_chi_Helmholtz}
with
\begin{align}\label{eq:F}
	F\coloneqq C-\frac{D^2}{4}-\frac{D_\zeta}{2}=C-\frac{D^2}{4}-\frac{A^{1/2}D'}{2}.
\end{align}
In order to find an explicit formula for $F$, we insert \eqref{eq:D} and
\[D'=-\frac{A'(B-\frac12A')}{2A^{3/2}}+\frac{B'-\frac12A''}{A^{1/2}}\]
into \eqref{eq:F} to get
\begin{align}
	F&=C-\frac{(B-\frac12A')^2}{4A}+\frac{A'(B-\frac12A')}{4A}-\frac{B'-\frac12A''}{2}\nonumber\\
	&=C-\frac{B^2}{4A}+\frac{BA'}{2A}-\frac{3A'^2}{16A}-\frac{B'}{2}+\frac{A''}{4}.\label{eq:F_2}
\end{align}
Now
\begin{align*}
	A'&=A^2(1-\mu)\left(\frac{\rho_0'}{p_0}u_0^2+g\frac{\rho_0^2}{p_0^2}u_0^2+2\frac{\rho_0}{p_0}u_0u_0'\right),\\
	A''&=\frac{2A'^2}{A}+A^2(1-\mu)\left(\frac{\rho_0''u_0^2}{p_0}+\frac{3g\rho_0\rho_0'u_0^2}{p_0^2}+\frac{4\rho_0'u_0u_0'}{p_0}+\frac{2g^2\rho_0^3u_0^2}{p_0^3}\right.\\
	&\qquad\qquad\qquad\qquad\qquad\left.{}+\frac{4g\rho_0^2u_0u_0'}{p_0^2}+\frac{2\rho_0u_0'^2}{p_0}+\frac{2\rho_0u_0u_0''}{p_0}\right)
\end{align*}
and
\begin{align*}
	B={}&A^2\left(\frac{\rho_0'}{\rho_0}+\frac{2(1-\mu)\rho_0u_0u_0'}{p_0}+\frac{g(1-\mu)\rho_0^2u_0^2}{p_0^2}\right),\\
	B'={}&\frac{2A'B}{A}+A^2\left[\frac{\rho_0''}{\rho_0}-\frac{\rho_0'^2}{\rho_0^2}+2(1-\mu)\left(\frac{\rho_0'u_0u_0'}{p_0}+\frac{\rho_0u_0'^2}{p_0}+\frac{\rho_0u_0u_0''}{p_0}\right.\right.\\
	&\qquad\qquad\quad\left.\left.{}+\frac{2g\rho_0^2u_0u_0'}{p_0^2}+\frac{g\rho_0\rho_0'u_0^2}{p_0^2}+\frac{g^2\rho_0^3u_0^2}{p_0^3}\right)\right].
\end{align*}
Plugging these formulas together with \eqref{eq:C} and \eqref{eq:A} into \eqref{eq:F_2}, we find after some computations
\begin{align*}
	A^{-3}F={}&- \frac{\rho_0' u_0'}{\rho_0 u_0} - \frac{u_0''}{u_0} - \frac{g \rho_0'}{\rho_0 u_0^2} + \frac{\rho_0'^2}{4 \rho_0^2} - \frac{\rho_0''}{2 \rho_0} - \frac{(1-\mu) \rho_0'^2 u_0^2}{p_0 \rho_0} + \frac{3 (1-\mu) \rho_0'' u_0^2}{4 p_0}\\
	& - \frac{(1-\mu) \rho_0' u_0 u_0'}{p_0} - \frac{5 (1-\mu) \rho_0 u_0'^2}{2 p_0} + \frac{3 (1-\mu) \rho_0 u_0 u_0''}{2 p_0} + \frac{ g (1-\mu) \rho_0'}{p_0}\\
	& - \frac{2 g (1-\mu) \rho_0 u_0'}{p_0 u_0} - \frac{g^2 (1-\mu) \rho_0}{p_0 u_0^2} + \frac{5 (1-\mu) ^2 \rho_0'^2 u_0^4}{16 p_0^2}\\
	& - \frac{(1-\mu) ^2 \rho_0 \rho_0'' u_0^4}{4 p_0^2} + \frac{(1-\mu) ^2 \rho_0 \rho_0' u_0^3 u_0'}{4 p_0^2} + \frac{3 (1-\mu) ^2 \rho_0^2 u_0^2 u_0'^2}{4 p_0^2}\\
	& - \frac{(1-\mu) ^2 \rho_0^2 u_0^3 u_0''}{2 p_0^2} - \frac{5 g (1-\mu) \rho_0 \rho_0' u_0^2}{4 p_0^2} - \frac{2 g \mu (1-\mu) \rho_0^2 u_0 u_0'}{p_0^2}\\
	& + \frac{g^2 (1-3\mu+2\mu^2) \rho_0^2}{p_0^2} + \frac{3 g (1-\mu) ^2 \rho_0^2 \rho_0' u_0^4}{8 p_0^3} + \frac{g (1-\mu) ^2 \rho_0^3 u_0^3 u_0'}{4 p_0^3}\\
	& + \frac{g^2 (1-\mu) \left(-\frac12 + \mu - \mu^2\right) \rho_0^3 u_0^2}{p_0^3} + \frac{g^2 (1-\mu) ^2 \rho_0^4 u_0^4}{16p_0^4}.
\end{align*}
Similarly, one could also derive an expression in terms of $(u_0,T_0)$, which we however omit. Above, note that everything (including derivatives $'$) in this formula is written with $z$ as the variable.

One might also be interested in a general explicit formula for $E$, which relates $\chi=E\bar w$ to $\bar w$. Writing out \eqref{eq:D}, we have
\[A^{-3/2}D=\frac{\rho_0'}{\rho_0}+(1-\mu)u_0\left(-\frac{\rho_0'u_0}{2p_0}+\frac{\rho_0u_0'}{p_0}+\frac{g\rho_0^2u_0}{2p_0^2}\right)\]
and thus
\begin{align*}
E&=\exp\int_0^z\frac{A^{-1/2}D}{2}\,\dd z'\\
&=\exp\int_0^z\frac{\frac{\rho_0'}{2\rho_0}+(1-\mu)u_0\left(-\frac{\rho_0'u_0}{4p_0}+\frac{\rho_0u_0'}{2p_0}+\frac{g\rho_0^2u_0}{4p_0^2}\right)}{1-(1-\mu)\frac{\rho_0}{p_0}u_0^2}\,\dd z',
\end{align*}
again written in terms of $z$.

We emphasize that the derivation of \eqref{eq:PDE_bar_w_original} subject to \eqref{eq:ABC} has not used any simplifying assumptions. A typical form of the Scorer parameter $F$ is depicted in Fig.~\ref{fig:Scorer_parameter}. In the meteorological research literature (see the surveys \cite{Smith79, Smith19})  some terms are neglected in order to simplify the expressions further. The most frequently used approximations are:
\begin{itemize}
	\item The first approximation is $\rho_0/p_0\approx0$, on the basis that $\rho_0/p_0=1/T_0$ with non-dimensional order of magnitude $U'^2/({\mathfrak R}'T') \approx 3 \times 10^{-4}$, given the already specified 
	reference values of $U'$ and ${\mathfrak R}'$ and the fact that the 
	minimal temperature of the troposphere varies by location and season, but is typically in the range from $-50\,^\circ$C to $-80\,^\circ$C (corresponding to $223.15\,^\circ$K to $193.15\,^\circ$K). Then
	\begin{gather*}
		A\approx1,\quad\zeta\approx z,\quad B\approx\frac{\rho_0'}{\rho_0},\\
		C\approx-\frac{\rho_0'u_0'}{\rho_0u_0}-\frac{u_0''}{u_0}-\frac{g\rho_0'}{\rho_0u_0^2}\approx-\frac{\rho_0'u_0'}{\rho_0u_0}-\frac{u_0''}{u_0}+\frac{g\beta}{u_0^2},\\
		E\approx\left(\frac{\rho_0}{\rho_0(0)}\right)^{1/2},\;
		F\approx-\frac{\rho_0'u_0'}{\rho_0u_0}-\frac{u_0''}{u_0}+\frac{g\beta}{u_0^2}-\frac{\rho_0'^2}{4\rho_0^2}-\frac12\left(\frac{\rho_0'}{\rho_0}\right)',
	\end{gather*}
	where the latter is the classical form of the Scorer parameter $F$. 
	\item A second, additional approximation is Boussinesq, that is, $\rho_0'\approx0$ in all terms without $g$. Here, one has
	\[A\approx B\approx1,\quad\zeta\approx z,\quad C\approx F\approx-\frac{u_0''}{u_0}+\frac{g\beta}{u_0^2},\]
	where the latter is another form of the Scorer parameter often found in the meteorological research literature \cite{Smith79}. Note that in the Boussinesq approximation the terms involving density variations in the flow are considered to be much smaller than the inertial terms and are ignored everywhere except in the buoyancy. However, in cases of realistic atmospheric multi-layered flows, the density gradients separating the homogeneous layers can be large and such an approximation is not reasonable 
	(see \cite{ba, w}). Moreover, the comprehensive study \cite{gg} shows that the Boussinesq regime applies only if the atmospheric temperature variations do not exceed $29\,^\circ$C. The troposphere is heated from below by the Earth's infrared radiation and convective processes are responsible for the decrease of temperature with altitude, at an average rate of about $0.65\,^\circ$C per 100 m. Near the tropopause this decrease 
	reaches $50\,^\circ$C at mid-latitudes and often exceeds $80\,^\circ$C in the tropics. Note that the centripetal acceleration due to the Earth's rotation about its polar axis causes 
	the top of the troposphere to drop poleward, from about 18 km at the Equator to about 6 km at the poles (see \cite{cj, va}). At mid-latitudes 
	the tropopause level also fluctuates quite strongly in time, being located as low as 10 km one day and as high as 15 km a week later (see \cite{bi, ho}).
\end{itemize}

\begin{figure}
	\includegraphics{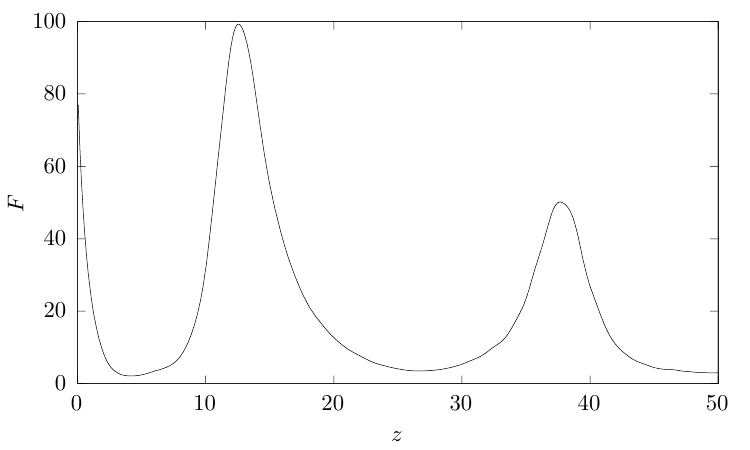}
	\caption{The annual average of the Scorer parameter $F$ at latitude $45\,^\circ$N, with $u_0$ the zonal wind component. Computed with data from CIRA-86 \cite{CIRA86_2,CIRA86}. Here, $A=1$ up to an error of $10^{-2}$, so that $z=\zeta$, with $z=50$ corresponding to an altitude of $50 L' = 100\;\mathrm{km}$.}
	\label{fig:Scorer_parameter}
\end{figure}

To conclude, the linear problem \eqref{eq:system_linear} reduces to \eqref{eq:PDE_chi_Helmholtz} with boundary condition \eqref{eq:kinematic_BC_linear}, that is,
\begin{subequations}
	\begin{align}
		\chi_{xx}+\chi_{\zeta\zeta}+F\chi &= 0 & \text{in }\R^2_+,\\
		\chi &= \bar h_x u_0 & \text{on }\partial\R^2_+,
	\end{align}
\end{subequations}
recalling that $E(0)=1$ by \eqref{eq:E}.

\section{Formal derivation of the solution formula}\label{sec:derivation}
\subsection{Some physical considerations}
For the ease of notation, let us return to $(x,z)$ as the variables on the upper half-plane, rename $\chi$ as $w$ (not to confuse with the $w$ in \eqref{eq:system_nonlinear}), and denote $f\coloneqq\bar h_x u_0(0)$. Thus, we shall now study mathematically the problem
\begin{subequations}\label{eq:BVP}
	\begin{align}
		\Delta w+Fw &= 0 & \text{in }\R^2_+,\label{eq:BVP_PDE}\\
		w &= f & \text{on }\partial\R^2_+,\label{eq:BVP_BC}
	\end{align}
\end{subequations}
for given $f$. In reality, the Scorer parameter $F$ is typically positive, as is also illustrated in Fig.~\ref{fig:Scorer_parameter}. This is indeed the interesting case which supports waves, while a negative $F$ would lead to unique, in all directions exponentially decaying solutions of \eqref{eq:BVP} (assuming decaying $f$). In fact, for this only the behavior of $F$ near infinity is relevant. Therefore and for mathematical purposes later on, in particular to put hands on the Dirichlet spectrum of the associated Schrödinger operator on the half line (see Section~\ref{sec:rigorous}), we shall always assume
\begin{align}
	\text{there exists }F_0 > 0 \text{ such that }& F-F_0, F' \in L^1([0,\infty))\nonumber\\
	\text{and }&\int_0^\infty z|F(z)-F_0|\,\dd z<\infty.\tag{A}\label{ass:A}
\end{align}
In physical terms, \eqref{ass:A} entails that $u_0$ has a fixed (positive) sign, that is, the background flow does not change its direction (at such points $F$ would blow up, containing terms with $u_0$ in the denominator). Moreover, \eqref{ass:A} implies that at large altitudes \eqref{eq:BVP_PDE} behaves like a usual Helmholtz equation. In fact, it is reasonable to assume \eqref{ass:A} in view of Fig.~\ref{fig:Scorer_parameter} and because at even higher altitudes the Horizontal Wind Model \cite{HWM14} suggests nearly constant $u_0$, while typically $\rho_0$ and $p_0$ decay with the same exponential rate. Thus, $F\approx-\frac{u_0''}{u_0}+\frac{g\beta}{u_0^2}$ is roughly constant at high altitudes.

In the context of the Helmholtz equation on the upper half-plane, it is well-known that there is an inherent non-uniqueness in the corresponding boundary value problem, due to radiating waves. This means that some physical consideration has to single out the correct physical solution among all mathematically reasonable ones. The mathematical literature in this direction, however, is restricted to the case of electromagnetic or acoustic waves. Here, various radiation conditions have been established (see \cite{Schot92} for an overview and \cite{BonnetEtAl09,CiraoloMagnanini09,Xu95} and the references therein for more recent results), all of them ultimately building up on Sommerfeld's radiation condition \cite{Sommerfeld1912}. The physical intuition behind this is that energy must radiate from sources to infinity and not vice versa. Now, in the electromagnetic or acoustic context, waves are emitted radially from a source. This leads ultimately to the fact that a Green's function associated to \eqref{eq:BVP} will be symmetric in $x$.

However, this can clearly not be correct for mountain waves -- it is obvious that waves should be found downstream and not upstream. This natural asymmetry, conversely, has to be taken care of when selecting the physically relevant solution of \eqref{eq:BVP} in the context of mountain waves and formulating a radiation condition. In fact, how to do this, even in the simplest case of constant $F$ and explicit, \enquote{nice} $f$, sparked some controversy in the meteorological research literature between the 1950s and 1970s; see the discussion in \cite{Smith79}. There are now some accepted physical arguments \cite{Corby54,EliassenPalm54,EliassenPalm60,Queney47}, and for the purpose of illustration especially the one by Lyra \cite{Lyra43} is relevant for us.

In the context of constant $F$, Lyra explained physically that a correct Green's function should
\begin{enumerate}[label=(L\arabic*),ref=(L\arabic*),leftmargin=*]
	\item \emph{ensure a greatest possible cancellation of waves windward and a largest possible reinforcement of waves leeward},\label{cond:L1}
	\item \emph{be strictly monotone in $x$ windward, for fixed $z$, ensuring the same property for $w$ as well}.\label{cond:L2}
\end{enumerate}
We shall now construct the physically relevant solution to \eqref{eq:BVP} according to these principles.

\subsection{The transform method}
Here, the strategy is to use a transform method based on the underlying Schrödinger operator
\[L\coloneqq -\frac{\dd^2}{\dd z^2} + q,\qquad \text{where } q\coloneqq F_0-F,\]
on the half line. This method is based on classical Weyl--Titchmarsh theory \cite{CoddingtinLevinson,Titchmarsh46} and has been somewhat similarly used in the context of wave propagation in stratified media \cite{BonnetEtAl09,DermenjianGuillot,MagnaniniSantosa01,Weder91,Wilcox84}. This section is only concerned with the formal computations; we postpone the rigorous analysis to Section~\ref{sec:rigorous}. In particular, we shall always tacitly assume that all quantities are sufficiently nice in order to justify all computations.

For $\lambda\in\R$, $v(\cdot,\lambda)$ is defined to be the solution of
\begin{subequations}\label{eq:def_v}
	\begin{gather}
		Lv = \lambda v,\quad \text{that is,}\quad v''+(\lambda+F-F_0)v=0\quad\text{on }[0,\infty),\label{eq:def_v_ODE}\\
		v(0)=0,\quad v'(0)=1,\label{eq:def_v_IC}
	\end{gather}
\end{subequations}
For $q\in L^1([0,\infty))$, it is well-known that $L$ is of limit-point type at infinity and its Dirichlet spectrum is comprised by the continuous part $[0,\infty)$ (with no embedded eigenvalues) and potentially a discrete part below $0$. Now let $\eta$ be the associated spectral measure. By classical theory, $\eta$ has mass on at most $[F_0-F_*,\infty)$, where
\[F_*\coloneqq \sup_{[0,\infty)} F<\infty,\]
and is absolutely continuous with respect to the Lebesgue measure on $[0,\infty)$; we write
\begin{equation}\label{eq:measure_continuous}
	\dd\eta(\lambda) = \sigma(\lambda)\, \dd\lambda, \quad \lambda\ge 0.
\end{equation}
Now, introduce
\begin{equation}\label{eq:def_U}
	U(x,\lambda)\coloneqq\int_0^\infty w(x,z)v(z,\lambda)\,\dd z.
\end{equation}
Multiplying \eqref{eq:BVP_PDE} by $v(z,\lambda)$, integrating in $z$, recalling \eqref{eq:def_v} and using the boundary condition \eqref{eq:BVP_BC} yields
\begin{align*}
	0&=\int_0^\infty(\Delta w(x,z)+F(z)w(x,z))v(z,\lambda)\,\dd z\\
	&=\int_0^\infty(w_{xx}(x,z)v(z,\lambda)+(v''(z,\lambda)+F(z)v(z,\lambda))w(x,z))\,\dd z\\
	&\quad\;+w(x,0)v'(0,\lambda)\\
	&=U_{xx}(x,\lambda)+(F_0-\lambda)U(x,\lambda)+f(x).
\end{align*}
That is, $U=U(\cdot,\lambda)$ solves the ordinary differential equation
\begin{equation}\label{eq:ODE_U}
	U_{xx}+(F_0-\lambda)U=-f\quad\text{in }\R.
\end{equation}
We now have to distinguish the two cases $\lambda>F_0$ and $\lambda<F_0$ (the single value $\lambda=F_0$ being irrelevant due to the absence of embedded eigenvalues). This is precisely the dichotomy between waves decaying at infinity and waves radiating energy to infinity. More precisely, for $\lambda>F_0$, the only bounded solution of \eqref{eq:ODE_U} is
\begin{equation}\label{eq:formula_U_evenescent}
	U(x,\lambda)=\int_\R\frac{e^{-\sqrt{\lambda-F_0}|x-x'|}}{2\sqrt{\lambda-F_0}}f(x')\,\dd x',
\end{equation}
so that we do not have any other choice in this case. For $\lambda<F_0$ the solutions of \eqref{eq:ODE_U} are oscillatory, and we have to find the physically correct solution. This is exactly the point where a careful physical argument or radiation condition comes into play. Here, in the context of electromagnetic or acoustic waves, classically the radial, \enquote{outgoing} kernel
\[\frac{e^{i\sqrt{F_0-\lambda}|\cdot|}}{2i\sqrt{F_0-\lambda}}\]
is the correct choice. As we have already observed, this choice cannot be correct in the context of mountain waves. Instead, we recall the principle \ref{cond:L1}. In fact, there is precisely one kernel for \eqref{eq:ODE_U} which vanishes identically upstream and moves all mass downstream, and therefore guarantees \ref{cond:L1}; this kernel is
\[\frac{\sin(\sqrt{F_0-\lambda}\;\cdot)1_{(0,\infty)}}{\sqrt{F_0-\lambda}},\]
with $1_{(0,\infty)}$ being the characteristic function of $(0,\infty)$. In other words, we are led to select
\[U(x,\lambda)=-\int_{-\infty}^x\frac{\sin(\sqrt{F_0-\lambda}(x-x'))}{\sqrt{F_0-\lambda}}f(x')\,\dd x'\]
as the physically correct solution of \eqref{eq:ODE_U}. Here, we see that these $U$ vanish below the support of $f$, that is, windward before the onset of the mountain range. As a consequence, for large negative $x$, only the decaying solutions in \eqref{eq:formula_U_evenescent} will contribute to $w$, and the monotonicity of the kernel there forces windward monotonicity of $w$ as well. This is basically the mechanism which leads to the principle \ref{cond:L2}. We shall return to this and make this more precise in Section~\ref{sec:rigorous}.

Having found now all $U$'s, we can reconstruct $w$. This is done by applying the inverse transform \cite{Titchmarsh46} to \eqref{eq:def_U}, that is,
\begin{align}
	w(x,z) &= \int_\R U(x,\lambda)v(z,\lambda)\,\dd\eta(\lambda)\nonumber\\
	&=\int_{F_0}^\infty \int_\R\frac{e^{-\sqrt{\lambda-F_0}|x-x'|}}{2\sqrt{\lambda-F_0}}f(x')\,\dd x'\,v(z,\lambda)\,\dd\eta(\lambda)\nonumber\\
	&\quad\; -\int_{F_0-F_*}^{F_0}\int_{-\infty}^x\frac{\sin(\sqrt{F_0-\lambda}(x-x'))}{\sqrt{F_0-\lambda}}f(x')\,\dd x'\,v(z,\lambda)\,\dd\eta(\lambda)\nonumber\\
	&=(K(\cdot,z)\ast f)(x),\label{eq:formula_w}
\end{align}
where the kernel $K$ is
\begin{align*}
K(x,z) &\coloneqq \int_{F_0}^\infty\frac{e^{-\sqrt{\lambda-F_0}|x|}}{2\sqrt{\lambda-F_0}}v(z,\lambda)\,\dd\eta(\lambda)\\
&\quad\;-1_{x>0}\int_{F_0-F_*}^{F_0}\frac{\sin(\sqrt{F_0-\lambda}x)}{\sqrt{F_0-\lambda}}v(z,\lambda)\,\dd\eta(\lambda).
\end{align*}
This is already the desired solution formula. To illustrate the kernel physically in more detail, we split the second part above into negative and positive $\lambda$, recall \eqref{eq:measure_continuous}, and arrive at three pieces which contribute to
\begin{equation}\label{eq:K_split}
	K = K^{\mathrm{e}} + K^{\mathrm{r}} + K^{\mathrm{t}}.
\end{equation}
\begin{itemize}
	\item \emph{The evanescent piece.} In case $\lambda > F_0$, the modes decay exponentially in $|x|$ and are bounded and oscillatory in $z$. We have
	\begin{equation}\label{eq:K^e}
		K^{\mathrm{e}}(x,z) = \int_{F_0}^\infty\frac{e^{-\sqrt{\lambda-F_0}|x|}}{2\sqrt{\lambda-F_0}}v(z,\lambda)\sigma(\lambda)\,\dd\lambda.
	\end{equation}
	\item \emph{The radiated piece.} In case $\lambda \in (0,F_0)$, the modes are bounded and oscillatory in both $x$ and $z$. This piece corresponds precisely to what is known as \emph{vertically propagating waves} \cite{Smith79}. We have
	\begin{equation}\label{eq:K^r}
		K^{\mathrm{r}}(x,z) = -1_{x>0}\int_0^{F_0}\frac{\sin(\sqrt{F_0-\lambda}x)}{\sqrt{F_0-\lambda}}v(z,\lambda)\sigma(\lambda)\,\dd\lambda.
	\end{equation}
	\item \emph{The trapped piece.} This comprises the potential Dirichlet eigenvalues $\lambda\in [F_0-F_*,0)$ of $L$, with modes that are bounded and oscillatory in $x$ and decaying in $z$ (in the sense of membership in $L^2$ with respect to $z$). This piece corresponds precisely to what is known as \emph{trapped lee waves} \cite{Smith79}. Denoting by $n\ge0$ the number of such eigenvalues, which we label as $\lambda_j$, $j=1,\ldots,n$, we have
	\begin{equation}\label{eq:K^t}
		K^{\mathrm{t}}(x,z) = - 1_{x>0}\sum_{j=1}^n \frac{\sin(\sqrt{F_0-\lambda_j}x)}{\sqrt{F_0-\lambda_j}}\frac{v(z,\lambda_j)}{\|v(\cdot,\lambda_j)\|_{L^2([0,\infty))}^2}.
	\end{equation}
\end{itemize}
We identify the extent of horizontal propagation as a key feature that distinguishes vertically propagating waves from trapped lee waves: On the level of $K$, the amplitude of vertically propagating waves tends to zero downstream by the Riemann--Lebesgue lemma, while the amplitude of trapped lee waves remains constant at fixed altitude.

Since trapped lee waves are of special physical interest, let us provide (rigorous) lower and upper bounds for their number. These are merely rephrasings of the classical results on bound states in a central potential due to Bargmann \cite{Bargmann52} and Calogero \cite{Calogero56}. (Note that 3D bound states $\tilde v$ of $-\Delta + q(|\cdot|)$ are in one-to-one correspondence to 1D Dirichlet eigenstates $v$ of $L$ via $\tilde v(Z) = z^{-1} v(z)$, $z=|Z|$, $Z\in\R^3$.) There are also other, more complicated such bounds, which also apply in our setting; see, for example, \cite[Thm. XIII.9]{ReedSimonIV}.
\begin{proposition}\label{prop:trapped}
	Let $n\ge0$ as in \eqref{eq:K^t}. Then we have:
	\begin{enumerate}[label=(\alph*)]
		\item (Upper bound on $n$.)
		\[n \le \int_0^\infty z|F(z)-F_0|\,\dd z.\]
		\item (Lower bound on $n$.)
		\[n \ge \left\lfloor \frac12 + \int_0^\infty \frac{F(z)-F_0}{\pi M}\,\dd z \right\rfloor,\; \text{with }M>0,\, M\ge\sqrt{F_*-F_0},\]
		where $\lfloor\cdot\rfloor$ is the floor function.
	\end{enumerate}
\end{proposition}
\begin{remark}
	\begin{itemize}
		\item When convolving $K(\cdot,z)$ with $f$, some of these $n$ \enquote{potential} trapped lee waves might however get canceled out. They survive precisely when
		\[\int_{-\infty}^x \sin(\sqrt{F_0-\lambda_j}(x-x')) f(x')\, \dd x' \ne 0 \quad \text{leeward},\]
		which means, at least in the case of compactly supported $f$, that
		\[\hat f(\sqrt{F_0-\lambda_j}) \ne 0,\]
		where $\hat f$ is the usual Fourier transform of $f$.
		\item Let us emphasize the following interesting connection to the meteorological research literature: Assuming $F\ge F_0$, $F\not\equiv F_0$, and replacing $F$ by $aF$, $a>0$ (or by $aF-b$, $b\in[0,aF_0)$), and $F_0$, $F_*$ accordingly, the integral appearing in the lower bound above is strictly positive and scales, for optimal $M$, like $\sqrt a$. Thus, for $a$ large enough, there always exist (arbitrarily many) trapped lee waves (if they do not get canceled out in the sense above). This can be viewed as a mathematically precise statement of Scorer's condition on the existence of trapped lee waves: that the Scorer parameter $F$ decreases (to $F_0$) rapidly with height \cite{Scorer49}.
	\end{itemize}
\end{remark}

\section{Examples}\label{sec:examples}
Before we turn to a rigorous analysis of the derived solution formula \eqref{eq:formula_w}, we first provide instructive examples for $F$ and corresponding $K$. To this end, let us recall the classical formula for the spectral function \cite{Titchmarsh46}
\begin{equation}\label{eq:formula_sigma}
	\sigma(\lambda) = \lim_{z\to\infty}\frac{1}{\pi(v'(z,\lambda)^2/\sqrt\lambda+\sqrt\lambda v(z,\lambda)^2)},\quad \lambda>0.
\end{equation}

\subsection{\texorpdfstring{$\mathbf{F\equiv0}$}{F = 0}} 
Despite not supporting waves and therefore not being of uttermost importance in applications (and also violating \eqref{ass:A}, strictly speaking), we illustrate the transform method in the case of vanishing Scorer parameter, that is, in the case of the Laplace equation. In fact, we will see in Section~\ref{sec:rigorous} that the kernel here is always the singular part of the kernel for general $F$.

In this setting
\begin{equation}\label{eq:formula_v_F=0}
	v(z,\lambda) = \frac{\sin(\sqrt\lambda z)}{\sqrt\lambda}
\end{equation}
and thus
\begin{equation}\label{eq:formula_sigma_F=0}
	\sigma(\lambda) = \frac{\sqrt\lambda}{\pi}
\end{equation}
by \eqref{eq:def_v} and \eqref{eq:formula_sigma}. It is clear that $K=K_0$ is only made up of the evanescent piece, and we have
\begin{align}
	K_0(x,z) &= \int_0^\infty \frac{e^{-\sqrt\lambda|x|}}{2\sqrt\lambda}\frac{\sin(\sqrt\lambda z)}{\sqrt\lambda}\frac{\sqrt\lambda}{\pi}\,\dd\lambda = \frac{1}{\pi} \int_0^\infty e^{-\mu|x|}\sin(\mu z)\,\dd \mu \nonumber\\
	&= \frac{z}{\pi(x^2+z^2)}.\label{eq:K0}
\end{align}
Of course, $K_0$ is nothing else than the usual Poisson kernel on the upper half-plane.

\subsection{\texorpdfstring{$\mathbf{F=\text{constant}}$}{F = constant}}\label{sec:F=const}
A case often studied in the meteorological research literature, due to computational convenience, is the case of constant Scorer parameter $F>0$. Here, the kernel still does not have a trapped piece, but has now a radiated piece. Moreover, the formulas \eqref{eq:formula_v_F=0} and \eqref{eq:formula_sigma_F=0} are still valid. Thus,
\begin{align*}
	K(x,z) &= \frac{1}{2\pi} \int_F^\infty\frac{e^{-\sqrt{\lambda-F}|x|}}{\sqrt{\lambda-F}}\sin(\sqrt\lambda z)\,\dd\lambda\\
	&\quad\; - \frac{1_{x>0}}{\pi} \int_0^F\frac{\sin(\sqrt{F-\lambda}x)}{\sqrt{F-\lambda}}\sin(\sqrt\lambda z)\,\dd\lambda\\
	&= \frac1\pi \int_{\sqrt F}^\infty \frac{e^{-\sqrt{\gamma^2-F}|x|}}{\sqrt{\gamma^2-F}} \gamma \sin(\gamma z)\, \dd\gamma\\
	&\quad\; - \frac{2\cdot 1_{x>0}}{\pi} \int_0^{\sqrt F} \frac{\sin(\sqrt{F-\gamma^2}x)}{\sqrt{F-\gamma^2}} \gamma \sin(\gamma z)\, \dd\gamma\\
	&= \frac1\pi \int_0^\infty e^{-\mu|x|} \sin(\sqrt{\mu^2+F}z)\,\dd\mu\\
	&\quad\; - \frac{2\cdot 1_{x>0}}{\pi}\int_0^{\sqrt F} \sin(\mu x) \sin(\sqrt{F-\mu^2}z) \,\dd\mu.
\end{align*}
The second line is exactly the formula which Lyra \cite[Sec. II.8]{Lyra43} derived (in Lyra's notation, $F=k^2$, $K=2\frac{\partial F_1}{\partial z}$). We plot $K$ in Fig.~\ref{fig:KernelConstant}.

\begin{figure}
	\includegraphics{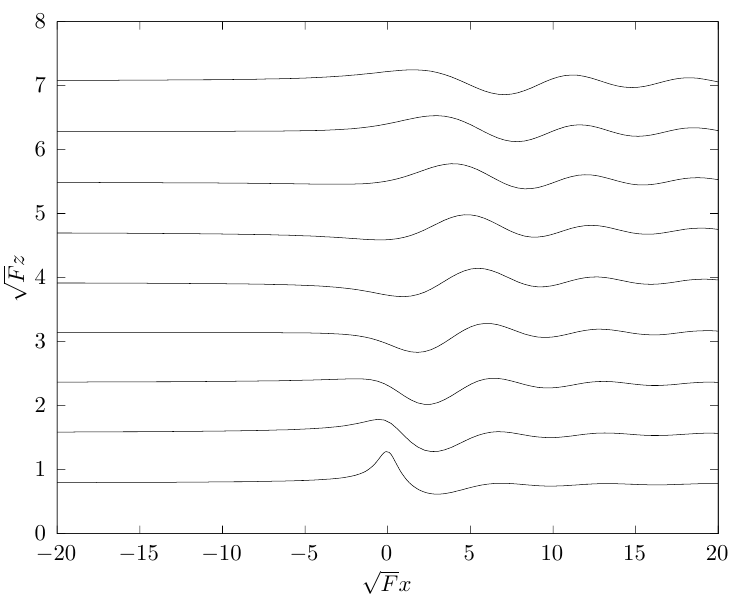}
	\caption{The kernel $K$ for constant $F$, plotted at various equidistant altitudes. Vertically propagating waves are visible, which are tilted backwards with height. The wave energy decreases rapidly downstream, and the wave pattern is essentially confined within a small region horizontally. See also \cite[Fig. 4]{Lyra43} for a similar illustration.}
	\label{fig:KernelConstant}
\end{figure}

An important aspect to notice here as well is the behavior of $K$ as $|x|\to\infty$, for fixed $z$. Notice that, integrating by parts,
\begin{align*}
	&\int_0^\infty e^{-\mu|x|}\sin(\sqrt{\mu^2+F}z)\,\dd\mu\\
	&=\frac{\sin(\sqrt{F} z)}{|x|}+\frac{z}{|x|}\int_0^\infty e^{-\mu|x|}\frac{\mu}{\sqrt{\mu^2+F}}\cos(\sqrt{\mu^2+F}z)\,\dd\mu\\
	&=\frac{\sin(\sqrt{F} z)}{|x|}+\mathcal O(|x|^{-3})
\end{align*}
and
\begin{align*}
	&\int_0^{\sqrt F}\sin(\mu x)\sin(\sqrt{F-\mu^2}z)\,\dd\mu\\
	&=\frac{\sin(\sqrt{F} z)}{x}-\frac{z}{x}\int_0^{\sqrt F}\cos(\mu x)\frac{\mu}{\sqrt{F-\mu^2}}\cos(\sqrt{F-\mu^2}z)\,\dd\mu\\
	&=\frac{\sin(\sqrt{F} z)}{x}+o(|x|^{-1}).
\end{align*}
In particular,
\[K(x,z)=-\frac{\sin(\sqrt{F} z)}{\pi x}+o(|x|^{-1}),\qquad |x|\to\infty,\]
and therefore
\[K(\cdot,z) \notin L^1(\R)\]
for all but countably many $z$. This is fundamentally different to the case of electromagnetic or acoustic waves, since there, due to the different formula for the radiated piece, the terms at order $|x|^{-1}$ exactly cancel out.

Turning to the principle \ref{cond:L2}, a similar sequence of integrations by parts shows that, windward,
\begin{align}
	K_x(x,z) &= \frac1\pi \int_0^\infty \mu e^{\mu x} \sin(\sqrt{\mu^2+F}z) \, \dd\mu \nonumber\\
	&= \frac{\sin(\sqrt{F} z)}{\pi x^2} + \frac{3z\cos(\sqrt{F}z)}{\pi\sqrt{F}x^4} + \mathcal O(|x|^{-6}),\quad x\to-\infty,\label{eq:Kx_asymptotics_F=const}
\end{align}
validating \ref{cond:L2}. We shall generalize this computation in Section~\ref{sec:rigorous}.

\subsection{\texorpdfstring{$\mathbf{F_0-F}=$ Morse potential}{F0 - F = Morse potential}}
Let us now have a look at an explicit example which allows for trapped lee waves and where $v$ and $\sigma$ can be computed explicitly in terms of confluent hypergeometric functions. (We are not aware of any such explicit example in the mountain wave literature.) We consider the well-known Morse potential \cite{Morse1929}
\begin{equation}\label{eq:Morse_potential}
	q(z) = Q\left(e^{-2a(z-z_0)} - 2e^{-a(z-z_0)}\right),
\end{equation}
with $Q>0$, $a>0$, $z_0\in\R$.

Let us first consider $\lambda>0$. By the substitution
\begin{equation}\label{eq:Kummer_substitution}
	v(z,\lambda) = \left(\frac{s}{2}\right)^{\mp i\sqrt{\lambda}/a} e^{-s/2} u^\pm(s,\lambda), \quad s = \frac{2\sqrt Q}{a} e^{-a(z-z_0)},
\end{equation}
\eqref{eq:def_v_ODE} is transformed into the Kummer equation
\[s\frac{\dd^2 u^\pm}{\dd s^2} + \left(\mp\frac{2i\sqrt{\lambda}}{a} + 1 - s\right) \frac{\dd u^\pm}{\dd s} - \left(\mp\frac{i\sqrt{\lambda}}{a} + \frac12 - \frac{\sqrt Q}{a}\right) u^\pm = 0.\]
Thus, we can write
\[v(z,\lambda) = \sum_\pm c^\pm \left(\frac{s}{2}\right)^{\mp i\sqrt{\lambda}/a} e^{-s/2} M\left(\mp\frac{i\sqrt{\lambda}}{a} + \frac12 - \frac{\sqrt Q}{a}, \mp\frac{2i\sqrt{\lambda}}{a} + 1, s\right)\]
for some $c^\pm\in\C$, where $M={}_1F_1$ is the Kummer function of the first kind. Now, \eqref{eq:def_v_IC} dictates the choice of $c^\pm$ by
\begin{align*}
	&\sum_\pm c^\pm \left(\frac{\sqrt Q}{a}\right)^{\mp i\sqrt\lambda/a} e^{\mp i\sqrt\lambda z_0} M\left(\mp\frac{i\sqrt{\lambda}}{a} + \frac12 - \frac{\sqrt Q}{a}, \mp\frac{2i\sqrt{\lambda}}{a} + 1, \frac{2\sqrt Q e^{az_0}}{a}\right)\\
	&= 0,\\
	&\sum_\pm c^\pm \Bigg[ \mp \frac{i\sqrt\lambda}{2\sqrt Q} \left(\frac{\sqrt Q}{a}\right)^{\mp i\sqrt\lambda/a} e^{(\mp i\sqrt\lambda-a) z_0}\\
	&\qquad\qquad\qquad\qquad \cdot M\left(\mp\frac{i\sqrt{\lambda}}{a} + \frac12 - \frac{\sqrt Q}{a}, \mp\frac{2i\sqrt{\lambda}}{a} + 1, \frac{2\sqrt Q e^{az_0}}{a}\right)\\
	&\quad\;+ \left(\frac{\sqrt Q}{a}\right)^{\mp i\sqrt\lambda/a} e^{\mp i\sqrt\lambda z_0}\left(\frac12-\frac{\sqrt Q}{a\mp 2i\sqrt\lambda}\right)\\
	&\qquad\qquad\qquad\qquad \cdot M\left(\mp\frac{i\sqrt{\lambda}}{a} + \frac32 - \frac{\sqrt Q}{a}, \mp\frac{2i\sqrt{\lambda}}{a} + 2, \frac{2\sqrt Q e^{az_0}}{a}\right)\Bigg]\\
	&= -\frac{\exp\left(\frac{\sqrt Q}{a}e^{az_0}-az_0\right)}{2\sqrt Q},
\end{align*}
where we used $M_s(a,b,s)=aM(a+1,b+1,s)/b$. This then yields an explicit solution formula for $v(z,\lambda)$ via \eqref{eq:Kummer_substitution}, which we however omit for simplicity.

Turning to $\sigma(\lambda)$, we shall use the Kodaira formula \cite[p. 940]{Kodaira49}
\begin{equation}\label{eq:Kodaira-}
	\sigma(\lambda)=\frac{\sqrt\lambda}{\pi|v_-(0,\lambda)|^2},
\end{equation}
where the Jost solution $v_-$ solves \eqref{eq:def_v_ODE} and behaves like $e^{-i\sqrt\lambda z}$ as $z\to\infty$. Since $s\to 0$ as $z\to\infty$ and $M(\cdot,\cdot,0)=1$ this translates to
\[c^+=0,\quad c^-=\left(\frac{\sqrt Q}{a}\right)^{-i\sqrt\lambda/a} e^{-i\sqrt\lambda z_0}.\]
Therefore,
\begin{align*}
	&|v_-(0,\lambda)|^2 \\
	&= \exp\left(-\frac{2\sqrt Q}{a} e^{az_0}\right) \left| M\left(\frac{i\sqrt{\lambda}}{a} + \frac12 - \frac{\sqrt Q}{a}, \frac{2i\sqrt{\lambda}}{a} + 1, \frac{2\sqrt Q e^{az_0}}{a}\right) \right|^2
\end{align*}
and hence
\begin{align*}
	\sigma(\lambda) = \frac{\sqrt\lambda}{\pi} \exp\left(\frac{2\sqrt Q}{a} e^{az_0}\right) \left| M\left(\frac{i\sqrt{\lambda}}{a} + \frac12 - \frac{\sqrt Q}{a}, \frac{2i\sqrt{\lambda}}{a} + 1, \frac{2\sqrt Q e^{az_0}}{a}\right) \right|^{-2}
\end{align*}
by \eqref{eq:Kodaira-}.

Let us now consider $\lambda<0$, that is, search for bound states of $L$. Here, the substitution is
\[v(z,\lambda) = \left(\frac{s}{2}\right)^{\sqrt{-\lambda}/a} e^{-s/2} u(s,\lambda), \quad s = \frac{2\sqrt Q}{a} e^{-a(z-z_0)},\]
and \eqref{eq:def_v_ODE} becomes
\begin{equation}\label{eq:Kummer}
	s\frac{\dd^2 u}{\dd s^2} + \left(\frac{2\sqrt{-\lambda}}{a} + 1 - s\right) \frac{\dd u}{\dd s} - \left(\frac{\sqrt{-\lambda}}{a} + \frac12 - \frac{\sqrt Q}{a}\right) u = 0.
\end{equation}
Notice that $M(\cdot,\cdot,0)=1$, but the Tricomi function $U(\frac{\sqrt{-\lambda}}{a} + \frac12 - \frac{\sqrt Q}{a}, \frac{2\sqrt{-\lambda}}{a} + 1,\cdot)$, which is another solution of \eqref{eq:Kummer} linearly independent of $M(\frac{\sqrt{-\lambda}}{a} + \frac12 - \frac{\sqrt Q}{a}, \frac{2\sqrt{-\lambda}}{a} + 1,\cdot)$, blows up at $s=0$. Thus, a bound state $v$ must correspond to
\[u(s,\lambda) = c M\left(\frac{\sqrt{-\lambda}}{a} + \frac12 - \frac{\sqrt Q}{a}, \frac{2\sqrt{-\lambda}}{a} + 1, s\right)\]
for some $c\in\R\setminus\{0\}$. In order to meet the Dirichlet boundary condition at $z=0$, the energy level $\lambda$ therefore has to satisfy the equation
\begin{equation}\label{eq:eigenvalue}
	M\left(\frac{\sqrt{-\lambda}}{a} + \frac12 - \frac{\sqrt Q}{a}, \frac{2\sqrt{-\lambda}}{a} + 1, \frac{2\sqrt Q e^{az_0}}{a}\right) = 0.
\end{equation}
Depending on the parameters $Q$, $a$, and $z_0$, we can have arbitrarily many bound states. Indeed, in the regime of large $2\sqrt Q e^{az_0}/a$, one can use asymptotics of $M$ as $s\to\infty$ to see that the energies distribute according to the law
\[\sqrt{-\lambda} \approx \sqrt Q - a\left(j-\frac12\right),\quad j=1,2,\ldots;\]
see the discussion in \cite{terHaar46}.

\begin{figure}
	\includegraphics{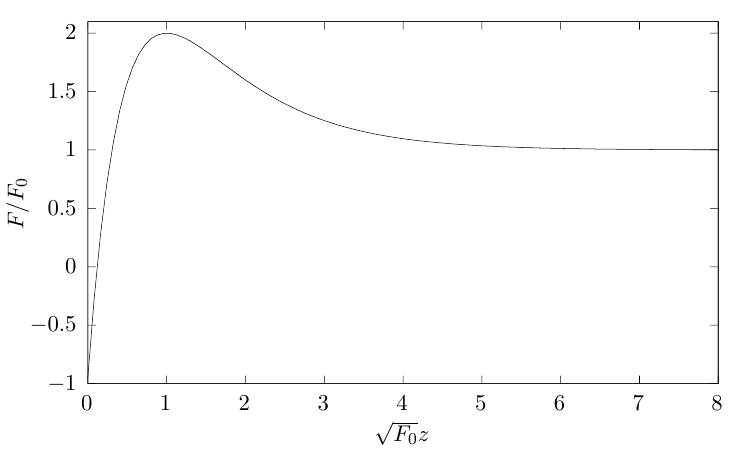}
	\caption{The Scorer parameter $F=F_0-q$, with $q$ and $F_0$ as in \eqref{eq:Morse_potential} and \eqref{eq:Morse_example}.}
	\label{fig:ScorerMorse}
\end{figure}

For illustration in a specific example, we choose parameters such that
\begin{equation}\label{eq:Morse_example}
	F_0=Q=a^2=z_0^{-2};
\end{equation}
see Fig.~\ref{fig:ScorerMorse} for the corresponding Scorer parameter. Notice that $\Delta+F$ is invariant, up a multiplicative constant, under the rescaling
\[(x,z,F_0,Q,a,z_0) \to (\kappa x,\kappa z,\kappa^{-2}F_0,\kappa^{-2}Q,\kappa^{-1}a,\kappa z_0),\quad \kappa>0,\]
so that the choice \eqref{eq:Morse_example} is up to rescaling effectively the same as the choice $F_0=Q=a=z_0=1$.

Now, \eqref{eq:eigenvalue} has exactly one solution $\lambda_1$ with
\[\frac{\lambda_1}{F_0}\approx-0.180,\]
Therefore, in the corresponding kernel $K$ we observe a trapped lee wave; see Fig.~\ref{fig:KernelMorse}. Notice that for the plot of $K$ we only need to compute $\lambda_1$ and the $\lambda$-integrals in \eqref{eq:K^e} and \eqref{eq:K^r} numerically.

\begin{figure}
	\includegraphics{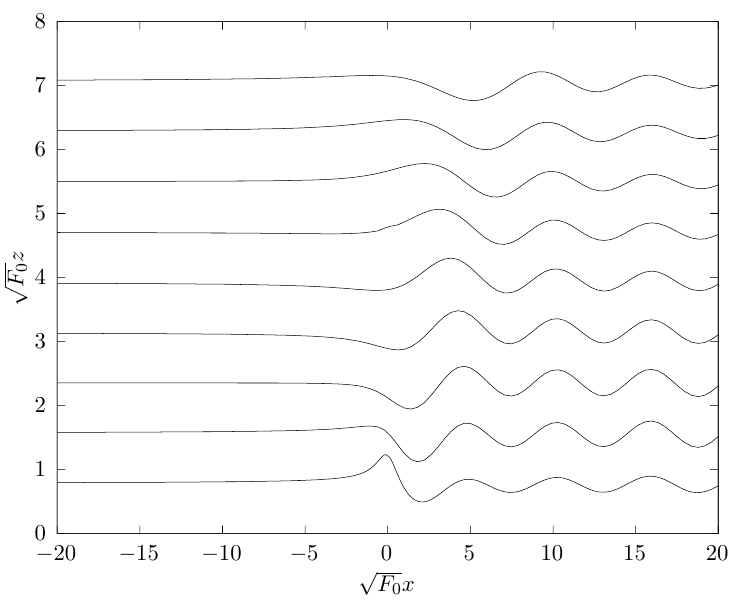}
	\caption{The kernel $K$, plotted at various equidistant altitudes, for $F=F_0-q$ (see Fig.~\ref{fig:ScorerMorse}), where $F_0>0$ and $q$ is the Morse potential in \eqref{eq:Morse_potential} with parameters as in \eqref{eq:Morse_example}. We recognize one trapped lee wave at frequency $\sqrt{F_0-\lambda_1}\approx 1.086\sqrt{F_0}$. The wave energy is trapped in a lower layer, and the wave propagates far downstream.}
	\label{fig:KernelMorse}
\end{figure}

\section{Rigorous mathematics}\label{sec:rigorous}
Finally, we clarify rigorously that, for arbitrary $F$ satisfying \eqref{ass:A}, the solution formula \eqref{eq:formula_w} is well-defined and provides indeed a solution to \eqref{eq:BVP} satisfying the principle \ref{cond:L2}. We shall also make clear for which boundary data $f$ this can be done.

To this end, we shall need the following lemma. Since we are not aware of a precise reference for all of these properties, we include a proof in Appendix~\ref{appx:proof}. Here and in the following, $c>0$ denotes some generic constant that may change from line to line.

\begin{lemma}\label{lem:spectral_results}
	Let $F$ satisfy \eqref{ass:A}. Then,
	\begin{align}
		\sigma &\in C((0,\infty)),\label{eq:sigma_continuous}\\
		\sigma(\lambda) &= \mathcal O\left(\frac{1}{\sqrt\lambda}\right),\quad \lambda\to 0,\label{eq:sigma_asymptotics_zero}\\
		\sigma(\lambda) &= \frac{\sqrt\lambda}{\pi}+\mathcal O\left(\frac{1}{\sqrt\lambda}\right),\quad \lambda\to\infty,\label{eq:sigma_asymptotics_infinity}\\
		|v(z,\lambda)| &\le \frac{c}{\sqrt\lambda},\label{eq:est_v}\\
		v(z,\lambda)-\frac{\sin(\sqrt\lambda z)}{\sqrt\lambda} &= -\frac{\cos(\sqrt\lambda z)}{2\lambda}\int_0^z q(z')\,\dd z'\, + \mathcal O(\lambda^{-3/2}),\quad \lambda\to\infty,\label{eq:asymp_v-v0}\\
		|v(z,\lambda)\sigma(\lambda)| &\le \begin{cases}c,&\lambda\ge1,\\\frac{c}{\sqrt\lambda},&0<\lambda\le1.\end{cases}\label{eq:est_v*sigma}
	\end{align}
	Here, all constants and error bounds are uniform in $\lambda$ and $z$. Furthermore, for $k\in\N_0$, if $F-F_0$ has $k+1$ moments, then
	\begin{equation}\label{eq:sigma_smooth}
		\sigma\in C^k((0,\infty)).
	\end{equation}
\end{lemma}

We can now prove good estimates of the kernel $K$. Let us first recall \eqref{eq:K_split}--\eqref{eq:K^t} and \eqref{eq:K0}. The first basic observation is that all integrals appearing in the formula for $K$ are (at least for $x\ne0$) well-defined by Lemma~\ref{lem:spectral_results}, but we can also establish a more precise statement.
\begin{lemma}\label{lem:K_est}
	Under Assumption \eqref{ass:A}, the singular part of $K$ is precisely $K_0$ in \eqref{eq:K0}, in the sense that $K-K_0$ is continuous on $\overline{\R^2_+}$ and
	\[|K(x,z)-K_0(x,z)| \le c\]
	with $c>0$ independent of $x$ and $z$.
\end{lemma}
\begin{proof}
	The deviation of the evanescent from the unperturbed piece can be written as follows:
	\begin{align*}
		&K^{\mathrm{e}}(x,z) - K_0(x,z) \\
		&= \int_{F_0}^\infty\frac{e^{-\sqrt{\lambda-F_0}|x|}}{2\sqrt{\lambda-F_0}}v(z,\lambda)\sigma(\lambda)\,\dd\lambda - \int_0^\infty \frac{e^{-\sqrt\lambda|x|}}{2\sqrt\lambda}\frac{\sin(\sqrt\lambda z)}{\sqrt\lambda}\frac{\sqrt\lambda}{\pi}\,\dd\lambda\\
		&= \int_{F_0}^\infty\left(\frac{e^{-\sqrt{\lambda-F_0}|x|}}{2\sqrt{\lambda-F_0}}-\frac{e^{-\sqrt\lambda|x|}}{2\sqrt\lambda}\right)v(z,\lambda)\sigma(\lambda)\,\dd\lambda\\
		&\quad\;+\int_{F_0}^\infty \frac{e^{-\sqrt\lambda|x|}}{2\sqrt\lambda}v(z,\lambda)\left(\sigma(\lambda)-\frac{\sqrt\lambda}{\pi}\right)\,\dd\lambda\\
		&\quad\;+\int_{F_0}^\infty \frac{e^{-\sqrt\lambda|x|}}{2\sqrt\lambda}\left(v(z,\lambda)-\frac{\sin(\sqrt\lambda z)}{\sqrt\lambda}\right)\frac{\sqrt\lambda}{\pi}\,\dd\lambda\\
		&\quad\;-\int_0^{F_0} \frac{e^{-\sqrt\lambda|x|}}{2\sqrt\lambda}\frac{\sin(\sqrt\lambda z)}{\sqrt\lambda}\frac{\sqrt\lambda}{\pi}\,\dd\lambda\\
		&\eqqcolon I_1+I_2+I_3+I_4.
	\end{align*}
	First, since $\dd(e^{-y}/y)/\dd y=-e^{-y}(y+1)/y^2$, we have
	\begin{align*}
		\left|\frac{e^{-\sqrt{\lambda-F_0}|x|}}{2\sqrt{\lambda-F_0}}-\frac{e^{-\sqrt\lambda|x|}}{2\sqrt\lambda}\right| &\le \frac{e^{-\sqrt{\lambda-F_0}|x|}(\sqrt{\lambda-F_0}|x|+1)}{2(\lambda-F_0)}\, (\sqrt\lambda-\sqrt{\lambda-F_0}) \\
		&\le c\lambda^{-3/2}
	\end{align*}
	for $\lambda\ge F_0+1$ independent of $x$. Thus and by \eqref{eq:est_v*sigma},
	\[|I_1| \le \int_{F_0}^{F_0+1} \frac{c}{\sqrt{\lambda-F_0}}\,\dd\lambda + c\int_{F_0+1}^\infty \lambda^{-3/2} \,\dd\lambda < \infty.\]
	Second, by \eqref{eq:sigma_asymptotics_infinity} and \eqref{eq:est_v},
	\[|I_2| \le c \int_{F_0}^\infty \lambda^{-3/2}\,\dd\lambda < \infty.\]
	Third, by \eqref{eq:est_v} and \eqref{eq:asymp_v-v0},
	\[\left|I_3 + \frac{1}{4\pi} \int_0^z q(z')\,\dd z' \int_{F_0}^\infty \frac{e^{-\sqrt\lambda|x|} \cos(\sqrt\lambda z)}{\lambda} \,\dd\lambda\right| \le c \int_{F_0}^\infty \lambda^{-3/2}\,\dd\lambda < \infty.\]
	Here, we realize that
	\[\int_1^\infty \frac{e^{-\sqrt\lambda|x|} \cos(\sqrt\lambda z)}{\lambda} \,\dd\lambda = 2\Real \int_1^\infty \frac{e^{\mu(-|x|+iz)}}{\mu} \,\dd\mu.\]
	It is well-known that this exponential integral has a logarithmic singularity near $(x,z)=(0,0)$ and, away from $(0,0)$, is bounded. Therefore and by $q\in L^1([0,\infty))\cap L^\infty([0,\infty))$,
	\[|I_3| \le c + c\left|\int_0^z q(z')\,\dd z' \, \Real \int_1^\infty \frac{e^{\mu(-|x|+iz)}}{\mu} \,\dd\mu\right| \le c.\]
	Fourth, clearly
	\[|I_4| \le c.\]
	Summing up, we have proved that
	\[|K^{\mathrm{e}}(x,z) - K_0(x,z)| \le c,\]
	where $c$ is independent of $x$ and $z$.
	
	The radiated and trapped pieces of $K$ are fortunately much easier to handle. Indeed, by \eqref{eq:est_v*sigma} we have
	\[|K^{\mathrm{r}}(x,z)|\le c\int_0^{F_0} \frac{\dd\lambda}{\sqrt{F_0-\lambda} \sqrt\lambda} <\infty\]
	and, by just being a finite sum,
	\[|K^{\mathrm{t}}(x,z)|\le c\]
	in view of \eqref{eq:est_v}.
	
	The continuity of $K-K_0$ follows directly from the above estimates.
\end{proof}

Now we are ready to prove our main theorem.
\begin{theorem}\label{thm:main}
	Let \eqref{ass:A} hold, $f\in C(\R)\cap L^1(\R)$, and $w$ be defined by \eqref{eq:formula_w}. Then:
	\begin{enumerate}[label=(\alph*)]
		\item\label{thm:main_a} (Solution of \eqref{eq:BVP}.) We have $w\in H^2_{\mathrm{loc}}(\R^2_+)$. Moreover, $w$ solves \eqref{eq:BVP_PDE} in the strong sense, and \eqref{eq:BVP_BC} in the sense that $w(x,z)\to f(x)$ as $z\to0$ for any fixed $x$.
		\item\label{thm:main_b} (Stability.) It holds that
		\[\|w\|_{L^\infty(\R^2)} \le c\|f\|_{L^1(\R)} + \|f\|_{L^\infty(\R)}.\]
		\item\label{thm:main_c} (Radiation condition.) Let $f$ be compactly supported in the negative $x$-direction and have three moments (w.r.t. $x$), and let $F-F_0$ have four moments (w.r.t. $z$). Then, for fixed $z$, $w(x,z)$ is monotone in $x$ for large negative $x$. More precisely,
		\begin{align}
			&w_x(x,z)\nonumber\\
			&= \frac{\int_\R f(x')\,\dd x'\, v(z,F_0) \sigma(F_0)}{x^2} + \frac{2\int_\R x'f(x')\,\dd x'\,v(z,F_0) \sigma(F_0)}{x^3} + o(|x|^{-3})\label{eq:wx_asymptotics}
		\end{align}
		as $x\to-\infty$.
	\end{enumerate}
\end{theorem}

\begin{remark}
	\begin{itemize}
		\item In \ref{thm:main_b}, we cannot expect an estimate with just the supremum norm of $f$ on the right-hand side since, by a similar argument as in Section~\ref{sec:F=const}, $K(\cdot,z)\notin L^1(\R)$ for all but countably many $z$.
		\item Formally, \eqref{eq:wx_asymptotics} only yields monotonicity for all but countably many $z$  provided the mean and the first moment of $f$ do not vanish simultaneously. If the latter happens to be true, under more regularity assumptions we could also compute the asymptotic expansion \eqref{eq:wx_asymptotics} up to order $m+2$, where $m$ is the smallest number such that the $m$-th moment of $f$ does not vanish. This would then also provide monotonicity.
		\item In the most important application of our theory, $f=\bar h_x u_0(0)$ is a derivative and vanishes at infinity. Therefore, its mean vanishes, but its first moment usually not (say, in the case of a symmetric mountain range and thus odd $f$). In this physically most relevant case, we have
		\[w_x(x,z) = \frac{2\int_\R x'f(x')\,\dd x'\,v(z,F_0) \sigma(F_0)}{x^3} + o(|x|^{-3})\]
		with nonzero leading order coefficient for all but countably many $z$.
	\end{itemize}
\end{remark}

\begin{proof}[Proof of Theorem~\ref{thm:main}]
	For \ref{thm:main_a}, first note that the convolution $w=K\ast f$ converges absolutely by Lemma~\ref{lem:K_est}. Also, by construction in Section~\ref{sec:derivation} it is clear that \eqref{eq:BVP_PDE} is satisfied, first at least in the sense of distributions, but then also in the strong sense by elliptic regularity. In particular, $w\in H^2_{\mathrm{loc}}(\R^2_+)$. Furthermore, by Lemma~\ref{lem:K_est} we can take the limit $z\to0$ under the integral to conclude
	\[\lim_{z\to0} ((K-K_0)(\cdot,z)\ast f)(x) = 0\]
	for fixed $x$, recalling $v(0,\lambda)=0$. Thus, the boundary condition \eqref{eq:BVP_BC} is satisfied by classical theory for the Poisson kernel $K_0$.
	
	Next, by Lemma~\ref{lem:K_est} and $\|K_0(\cdot,z)\|_{L^1(\R)}=1$, we have
	\begin{align*}
		&|w(x,z)|\\
		 &\le \int_\R |K(x-x',z)-K_0(x-x',z)| |f(x')|\,\dd x'\, + \int_\R K_0(x-x',z) |f(x')|\,\dd x'\\
		&\le c\|f\|_{L^1(\R)} + \|f\|_{L^\infty(\R)},
	\end{align*}
	which is \ref{thm:main_b}.
	
	As for \ref{thm:main_c}, let $R\in\R$ such that $f(x)=0$ for $x<R$. Then, for $x<R$,
	\begin{align*}
		w_x(x,z) &= \frac12\int_{F_0}^\infty e^{\sqrt{\lambda-F_0}x} \int_R^\infty e^{-\sqrt{\lambda-F_0}x'}f(x')\,\dd x'\,v(z,\lambda)\sigma(\lambda)\,\dd\lambda \\
		&= \int_0^1 \mu e^{\mu x} \int_\R e^{-\mu x'} f(x')\,\dd x'\, v(z,\mu^2+F_0) \sigma(\mu^2+F_0)\,\dd\mu + \mathcal O(e^x)
	\end{align*}
	Similarly as for \eqref{eq:Kx_asymptotics_F=const}, the general formula
	\begin{align*}
		\int_0^1 e^{\mu x}p(\mu)\,\dd\mu &= -\frac{p(0)}{x} + \frac{p'(0)}{x^2} - \frac{p''(0)}{x^3} - \frac{1}{x^3} \int_0^1 e^{\mu x} p'''(\mu)\,\dd\mu \,+ \mathcal O(e^x) \\
		&= -\frac{p(0)}{x} + \frac{p'(0)}{x^2} - \frac{p''(0)}{x^3} + o(|x|^{-3}),
	\end{align*}
	which holds for $p\in C^3([0,1])$ and  can easily be seen by successive integrations by parts, immediately implies \eqref{eq:wx_asymptotics}, noticing that we indeed have $\sigma\in C^3((0,\infty))$ by Lemma~\ref{lem:spectral_results}.
\end{proof}

\appendix

\section{Proof of Lemma~\ref{lem:spectral_results}}\label{appx:proof}
\begin{proof}
	We recall the Kodaira formula \cite[p. 940]{Kodaira49}
	\begin{equation}\label{eq:Kodaira}
		\sigma(\lambda)=\frac{\sqrt\lambda}{\pi|v_+(0,\lambda)|^2},
	\end{equation}
	where the Jost solutions $v_\pm$ solve \eqref{eq:def_v_ODE} and behave like $e^{\pm i\sqrt\lambda z}$ as $z\to\infty$. By \cite[Lem. 3.1.1]{Marchenko11} we know that
	\begin{equation}\label{eq:v+_kernel}
		v_+(0,\lambda) = 1 + \int_0^\infty G(0,z) e^{i\sqrt\lambda z}\,\dd z,
	\end{equation}
	for a certain kernel $G$ which satisfies
	\[|G(0,z)| \le \frac12 \int_{z/2}^\infty |q(t)|\,\dd t\, e^M\]
	with
	\[M\coloneqq\int_0^\infty z|q(z)|\,\dd z.\]
	(Here and in the following, note that $\sqrt\lambda$ is called $\lambda$ in \cite{Marchenko11}.) Clearly, $v_+(0,\cdot) \in C^k((0,\infty))$ provided $G(0,\cdot)$ has $k$ moments, since then one can differentiate under the integral sign in \eqref{eq:v+_kernel}. But
	\[\int_0^\infty z^k \int_{z/2}^\infty |q(t)|\,\dd t\dd z = \frac{2^{k+1}}{k+1} \int_0^\infty z^{k+1} |q(z)|\,\dd z.\]
	Moreover, in \cite{Kodaira49} it is shown that $v_+(0,\cdot)$ is continuous on $[0,\infty)$ and nonzero on $(0,\infty)$. Thus, \eqref{eq:sigma_continuous} and \eqref{eq:sigma_smooth} follow directly in view of \eqref{eq:Kodaira}.
	
	As for \eqref{eq:sigma_asymptotics_zero}, the behavior of $\sigma(\lambda)$ as $\lambda\to0$ is determined by the behavior of $v_+(0,\lambda)$: either $v_+(0,0)\ne0$ or $v_+(0,0)=0$ (the latter cannot be ruled out in general). If $v_+(0,0)\ne0$, then clearly $\sigma(\lambda)=\mathcal O(\sqrt\lambda)$ as $\lambda\to0$.
	If $v_+(0,0)=0$, however, we know at least that \cite[Lem. 3.1.6]{Marchenko11}
	\begin{equation}\label{eq:Jost_asymptotics_zero}
		\frac{\sqrt\lambda}{v_+(0,\lambda)} = \mathcal O(1), \quad \lambda\to0.
	\end{equation}
	Therefore, we have \eqref{eq:sigma_asymptotics_zero} in any case.
	
	Next, \eqref{eq:est_v} follows from \cite[Lem. 3.1.4]{Marchenko11}, and \eqref{eq:asymp_v-v0} is proved in \cite[Lem. 2]{Eastham97}. Turning to \eqref{eq:sigma_asymptotics_infinity}, we use the formula \cite{Titchmarsh46}
	\[\sigma(\lambda) = \frac{1}{\pi\sqrt\lambda(a(\lambda)^2+b(\lambda)^2)}\]
	for the usual Weyl--Titchmarsh coefficients
	\begin{align*}
		a(\lambda) &= -\frac{1}{\sqrt\lambda} \int_0^\infty \sin(\sqrt\lambda z) q(z) v(z,\lambda)\,\dd z,\\ b(\lambda) &= \frac{1}{\sqrt\lambda} + \frac{1}{\sqrt\lambda} \int_0^\infty \cos(\sqrt\lambda z) q(z) v(z,\lambda)\,\dd z.
	\end{align*}
	Here,
	\[a(\lambda) = \mathcal O(\lambda^{-1}),\quad b(\lambda) = \frac{1}{\sqrt\lambda} + \mathcal O(\lambda^{-3/2})\]
	by \cite[Lem. 3]{Eastham97}, and thus, \eqref{eq:sigma_asymptotics_infinity} follows immediately.
	
	Finally, \eqref{eq:est_v*sigma} follows directly from \eqref{eq:sigma_asymptotics_infinity} and \eqref{eq:asymp_v-v0} in case $\lambda\ge1$. For small $\lambda$, we have to take a detour and look at the Faddeev solution $\chi_\pm(z,\lambda)=e^{\mp i\sqrt\lambda z}v_\pm(z,\lambda)$, which solves the integral equation
	\[\chi_\pm(z,\lambda)=1\pm\int_z^\infty\frac{e^{\pm 2i\sqrt\lambda(z'-z)}-1}{2i\sqrt\lambda}q(z')\chi_\pm(z',\lambda)\,\dd z'.\]
	As in \cite[Pf. of Thm. 8.2.1]{BennewitzBrownWeikard20}, one can show by successive approximations that
	\[|\chi_\pm(z,\lambda)|\le e^M\]
	and thus
	\[|v_\pm(z,\lambda)|\le e^M\]
	as well. By the classical scattering relation (see, e.g., \cite[Lem. 3.1.5]{Marchenko11}) we have
	\[\frac{\sqrt\lambda v(z,\lambda)}{v_+(0,\lambda)}=\frac{S(\lambda)v_+(z,\lambda)-v_-(z,\lambda)}{2i},\]
	where
	\[S(\lambda) = \frac{v_-(0,\lambda)}{v_+(0,\lambda)} = \frac{\overline{v_+(0,\lambda)}}{v_+(0,\lambda)}\]
	is the scattering function. Thus, by \eqref{eq:Kodaira} and $|S(\lambda)|=1$,
	\[|v(z,\lambda)\sigma(\lambda)|=\frac{1}{\pi|v_+(0,\lambda)|}\cdot\left|\frac{S(\lambda)v_+(z,\lambda)-v_-(z,\lambda)}{2i}\right|\le\frac{e^M}{\pi|v_+(0,\lambda)|},\]
	from which \eqref{eq:est_v*sigma}, for small $\lambda$, follows in view of \eqref{eq:Jost_asymptotics_zero}.
\end{proof}

\paragraph{Acknowledgments}
This research was funded in whole or in part by the Austrian Science Fund (FWF) [10.55776/ESP8360524]. For open access purposes, the author has applied a CC BY public copyright license to any author-accepted manuscript version arising from this submission.

\bibliographystyle{amsplain}	
\bibliography{Scorer_ref}

\end{document}